\documentclass[11pt]{article}
\usepackage{a4wide}

\usepackage[a4paper, margin=2.3cm]{geometry}
\usepackage{amsmath,amssymb,amsthm}
\usepackage{color}
\usepackage{parskip}
\usepackage{hyperref}
\usepackage{parskip}
\usepackage{setspace}

\usepackage{caption}
\usepackage{subcaption}

\usepackage{graphicx}
\usepackage[nice]{nicefrac}
\usepackage{amssymb}
\usepackage{multirow}
\usepackage{array}

\usepackage{amsmath,amsthm,amscd,amssymb, mathrsfs}

\usepackage{latexsym}

\numberwithin{equation}{section}

\theoremstyle{plain}
\newtheorem{theorem}{Theorem}[section]
\newtheorem{lemma}[theorem]{Lemma}
\newtheorem{corollary}[theorem]{Corollary}
\newtheorem{proposition}[theorem]{Proposition}
\newtheorem{conjecture}[theorem]{Conjecture}
\newtheorem{question}[theorem]{Question}

\theoremstyle{definition}
\newtheorem{definition}[theorem]{Definition}

\newtheorem{problem}[theorem]{Problem}

\theoremstyle{remark}
\newtheorem{remark}[theorem]{Remark}

\newtheorem{notation}[theorem]{Notation}

\newtheorem{case[theorem]}{Case}

\title{Multi-linear forms, graphs, and $L^p$-improving measures in ${\Bbb F}_q^d$}
\author{Pablo Bhowmik\and  Alex Iosevich\and  Doowon Koh \and Thang Pham}
\begin{document} 

\maketitle

\begin{abstract} The purpose of this paper is to introduce and study the following graph theoretic paradigm. Let 
$$T_Kf(x)=\int K(x,y) f(y) d\mu(y),$$ where $f: X \to {\Bbb R}$, $X$ a set, finite or infinite,  and $K$ and $\mu$ denote a suitable kernel and  a measure, respectively. Given a connected ordered graph $G$ on $n$ vertices, consider the multi-linear form 
$$ \Lambda_G(f_1,f_2, \dots, f_n)=\int_{x^1, \dots, x^n \in X} \ \prod_{(i,j) \in {\mathcal E}(G)} K(x^i,x^j) \prod_{l=1}^n f_l(x^l) d\mu(x^l), $$ where 
${\mathcal E}(G)$ is the edge set of $G$. Define $\Lambda_G(p_1, \ldots, p_n)$ as the smallest constant $C>0$ such that the inequality
\begin{equation} \label{keyquestion} \Lambda_G(f_1, \dots, f_n) \leq C \prod_{i=1}^n {||f_i||}_{L^{p_i}(X, \mu)} \end{equation} holds for all non-negative real-valued functions $f_i$, $1\le i\le n$, on $X$. The basic question is, how does the structure of $G$ and the mapping properties of the operator $T_K$ influence the sharp exponents in (\ref{keyquestion}). In this paper, this question is investigated mainly in the case $X={\Bbb F}_q^d$, the $d$-dimensional vector space over the field with $q$ elements, and $K(x^i,x^j)$ is the indicator function of the sphere evaluated at $x^i-x^j$. In the sequel, we shall study this problem in the context of hypergraphs where the underlying object is a multi-linear operator in place of $T_K$. 

\end{abstract}

\tableofcontents

\section{Introduction} 

\vskip.125in 

One of the fundamental objects in harmonic analysis is the operator of the form 
\begin{equation} \label{euclideanoperator} T_Kf(x)=\int_{{\Bbb R}^d} K(x,y) f(y) dy, \end{equation} where $K: {\Bbb R}^d \times {\Bbb R}^d \to {\Bbb R}$ is a suitable kernel and $f$ is a locally integrable function. See \cite{St93} and the references contained therein for a variety of manifestations of operators of this type and their bounds. 

The purpose of this paper is to study operators from (\ref{euclideanoperator}) in the context of vector spaces over finite fields. Let ${\Bbb F}_q$ denote the finite field with $q$ elements and ${\Bbb F}_q^d$ the $d$-dimensional vector space over this field. Let $K: {\Bbb F}_q^d \times {\Bbb F}_q^d \to {\Bbb C}$ be a suitable kernel, and define 
$$ T_Kf(x)=\sum_{y \in {\Bbb F}_q^d} K(x,y) f(y).$$ 

Operators of this type have been studied before (\cite{CSW08, KL22, KS13, KSS16}).  In particular, the operator $T_K$ with $K(x,y)=S_t(x-y)$, where $S_t$ is the indicator function of the sphere 
$$ S_t=\{x \in {\Bbb F}_q^d: ||x||=t \},$$ $||x||=x_1^2+x_2^2+\dots+x_d^2$ comes up naturally in the study of the Erd\H{o}s-Falconer distance problem in vector spaces over finite fields, namely the question of how large a subset $E \subset {\Bbb F}_q^d$ needs to be to ensure that if 
$$ \Delta(E)=\{||x-y||: x,y \in E\},$$ then $|\Delta(E)| \ge \frac{q}{2}$. Here and throughout, $|S|$, with $S$ a finite set, denotes the number of elements in this set. See, for example (\cite{BHIPR, CEHIK09, HLR16, IR07, MPPRS}). 

If one is interested in studying more complicated geometric objects than distances, an interesting modification of the spherical averaging operator needs to be made. Indeed, let $E \subset {\Bbb F}_q^d$, and suppose that we want to know how many equilateral triangles of side-length $1$ it determines. The quantity that counts such triangles is given by 
\begin{equation} \label{equilateralform} \sum_{x,y,z \in {\Bbb F}_q^d} K(x,y)K(x,z)K(y,z)E(x)E(y)E(z),\end{equation} where $K(x,y)=S_1(x-y)$. 

Let us interpret the quantity (\ref{equilateralform}) in the following way. Let us view $x,y,z$ as vertices, and let us view the presence of $K(x,y)$ as determining the edge connecting $x$ and $y$, and so on. In this way, the quantity (\ref{equilateralform}) is associated with the graph $K_3$, the complete graph on three vertices (Figure \ref{fig:sub2}). 

Another natural example is the following. Let $K(x,y)=S_1(x-y)$, and consider the quantity that counts rhombi of side-length $1$, i.e.
\begin{equation} \label{rhombusform} \sum_{x,y,z,w\in {\Bbb F}_q^d} K(x,y)K(y,z)K(z,w)K(w,x)E(x)E(y)E(z)E(w). \end{equation} Arguing as above, we associate this form with the graph $C_4$, the cycle on four vertices (Figure \ref{fig:C4}). 

In general, let $K$ be a kernel function, and let $G$ be a connected ordered graph on $n$ vertices. Define 
\begin{equation} \label{generalform} \Lambda_G(f_1, f_2, \dots, f_n)=\frac{1}{\mathcal{N}(G)}\sum_{x^1, \dots, x^n \in {\Bbb F}_q^d} \ \prod_{(i,j) \in {\mathcal E}(G)} K(x^i,x^j) \prod_{l=1}^n f_l(x^l), \end{equation} where ${\mathcal E}(G)$ is the edge set of $G$ and $\mathcal{N}(G)$ is the normalizing factor defined as the number of distinct embeddings of $G$ in $\mathbb{F}_q^d$.

We note in passing that the paradigm we just introduced extends readily to the setting of hypergraphs. If we replace our basic object, the linear operator $T_K$, by an $m$-linear operator $M_K$, the problem transforms to the setting where the edges dictated by the kernel $K$ are replaced by hyperedges induced by the multi-linear kernel $K(x^1, \dots, x^{m+1})$. We shall address this formulation of the problem in the sequel. 

\vskip.125in 
The norm $||f||_p, 1\le p< \infty,$ is defined to be associated with normalizing counting measure on $\mathbb F_q^d.$ More precisely, given a function $f$ on $\mathbb F_q^d,$ we define
$$ {||f||}_p:={\left(q^{-d}\sum_{x \in {\Bbb F}_q^d} {|f(x)|}^p \right)}^{\frac{1}{p}}~~ (1\le p<\infty), \quad\mbox{and}\quad ||f||_\infty:=\max_{x\in \mathbb F_q^d} |f(x)|.$$ 

\begin{definition} \label{defIKT}Let $1\le p_i\le \infty$,  $i=1,\ldots, n,$ be integers. 
We define $\Lambda_G(p_1, \ldots, p_n)$ as the smallest constant $C>0$ such that  the following inequality
\begin{equation} \label{generalproblem} \Lambda_G(f_1, \dots, f_n) \leq C \prod_{i=1}^n {||f_i||}_{p_i} \end{equation}
holds for all non-negative real-valued functions $f_i$, $1\le i\le n$, on $\mathbb{F}_q^d$. \end{definition}
If there is a uniform constant $C\ge 1$ independent of $q$ such that $\Lambda_G(p_1, \ldots, p_n)\le C$, then we use the notation $\Lambda_G(p_1, \ldots, p_n) \lesssim 1$.

For the remainder of this paper, the kernel function $K(x, y)$ is assumed to be $S_t(x-y)$ with $t\ne 0.$ In addition,  when the dimension $d$ is two, we assume that the number $3$ in $ \mathbb F_q$ is a square number so that we can exclude the trivial case in which the shape of an equilateral triangle in $\mathbb F_q^2$ does not occur. 

The main purpose of this paper is to determine all numbers $1\le p_i\le \infty, i=1,2, \ldots, n,$ such that
$$ \Lambda_G(p_1, \ldots, p_n) \lesssim 1.$$
This problem will be called the boundedness problem for the operator $\Lambda_G$ on $\mathbb F_q^d.$

In this paper we shall mainly confine ourselves to the graphs $G$ which possess one of the following specific structures: $K_2$ (the graph with two vertices and one edge), $K_3$ (the cycle with three vertices and three edges), $K_3$ + tail (a kite), $P_2$ (the path of length $2$), $P_3$ (the path of length $3$), $C_4$ (the cycle with $4$ vertices and $4$ edges), $C_4$ + diagonal, $Y$-shape (a space station). It is worth noting that the study of the distribution of these graphs in a given set has been received much attention from people in Discrete Geometry and Geometric Measure Theory during the last decade, see for example \cite{DIOWZ, GIOW, GIT1, GIT2, KPV,LMH, OT, R}.
%%%%%%%%%%%%%%%%%%%%%%%%%%%%%%%%%%%%%%%%%%%%%%%%%
\begin{figure}[h!]
\centering
\begin{subfigure}{.4\textwidth}
  \centering
  \includegraphics[width=.3\linewidth]{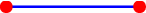}
  \caption{$G=K_2$}
  \label{fig:sub1}
\end{subfigure}%
\begin{subfigure}{.4\textwidth}
  \centering
  \includegraphics[width=.3\linewidth]{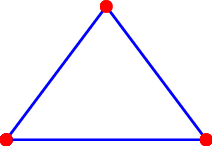}
  \caption{$G=K_3$}
  \label{fig:sub2}
\end{subfigure}
\label{fig:test}
\centering
\begin{subfigure}{.4\textwidth}
  \centering
  \includegraphics[width=.3\linewidth]{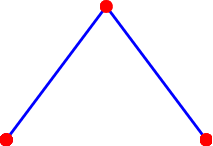}
  \caption{$G=P_2$}
\end{subfigure}%
\begin{subfigure}{.4\textwidth}
  \centering
  \includegraphics[width=.3\linewidth]{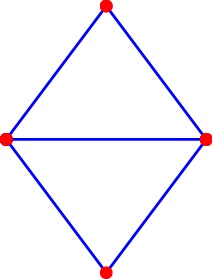}
  \caption{$G=C_4+$ diagonal}
   \label{fig:sub2K}
\end{subfigure}%
\label{fig:test}
\centering
\begin{subfigure}{.4\textwidth}
  \centering
  \includegraphics[width=.3\linewidth]{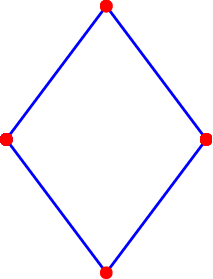}
  \caption{$G=C_4$}\label{fig:C4}
\end{subfigure}
\begin{subfigure}{.4\textwidth}
  \centering
  \includegraphics[width=.3\linewidth]{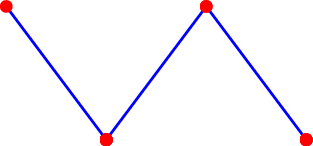}
  \caption{$G=P_3$}
\end{subfigure}
\label{fig:test}
\centering
\begin{subfigure}{.4\textwidth}
  \centering
  \includegraphics[width=.3\linewidth]{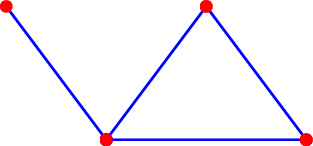}
  \caption{$G=K_3+$ tail}
  \label{fig:sub2KK}
\end{subfigure}
\begin{subfigure}{.4\textwidth}
  \centering
  \includegraphics[width=.3\linewidth]{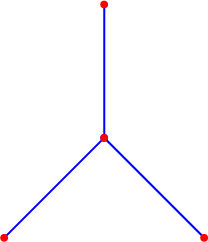}
  \caption{$G=Y$-shape}
  \label{fig:YS}
\end{subfigure}
\label{fig:test}
\end{figure}

%%%%%%%%%%%%%%%%%%%%%%%%%%%%%%%%%%%%%%%%%%%%%%%%%%%%%%%%%%%%
When the graph $G$ is the $K_2,$  the complete answer to the boundedness problem will be given in all dimensions. To deduce the result, we will invoke the spherical averaging estimates over finite fields (see Theorem \ref{GenRL}).

When the number of the vertices of the graph $G$ is $3$ or $4,$  we will obtain reasonably good boundedness results in two dimensions. In particular, in the case when the degree of each vertex is  at least two ($K_3, C_4+$ a diagonal, and $C_4$),  we shall prove  sharp results (up to the endpoints) for the operators on $\mathbb F_q^2$ (see Theorems \ref{Best2D}, \ref{KiteK}, and \ref{Thm7.6K}). For three and higher dimensions, the boundedness problem is not simple and we will  address partial results.

It is very natural to ask whether or not one can prove a general theorem that addresses all connected graphs on $n$ vertices. Unfortunately, such a result is beyond the scope of this paper. The main difficulties arise when the maximal degree is large or the edge set is dense, or if the graph contains a cycle or not. All of these issues will be illustrated in the proofs of our results.  

{\bf Notation:} Let $1\le p_i\le \infty $ and  $f_i, 1\le i\le n,$ be non-negative real-valued functions on $\mathbb F_q^d.$
\begin{itemize}
\item When $n=2$, we denote 
$ L(f_1, f_2):=\Lambda_G(f_1, f_2)$ and $L(p_1, p_2):=\Lambda_G(p_1, p_2).$
\end{itemize}

\begin{itemize}\item (Case of $n=3$) 
\begin{enumerate}
\item If $G$ is the graph $K_3$,  we denote
$\Delta(f_1, f_2, f_3):= \Lambda_G(f_1,f_2, f_3)$ and $\Delta(p_1, p_2, p_3):= \Lambda_G(p_1, p_2, p_3).$
\item  If $G$ is the graph $P_2$,  we denote
$\Lambda(f_1, f_2, f_3):= \Lambda_G(f_1, f_2, f_3)$ and $\Lambda(p_1, p_2, p_3):= \Lambda_G(p_1, p_2, p_3).$
\end{enumerate}
\end{itemize}

\begin{itemize}
\item (Case of $n=4$)
\begin{enumerate}
\item  If $G$ is the $C_4$ + diagonal, we denote
$\diamondsuit_t(f_1,f_2,f_3, f_4):=\Lambda_G(f_1, f_2, f_3, f_4)$ and\\
$\diamondsuit_t(p_1,p_2, p_3, p_4):=\Lambda_G(p_1, p_2, p_3, p_4).$

\item If $G$ is the $C_4$, we denote
$\Diamond(f_1,f_2, f_3, f_4):=\Lambda_G(f_1, f_2, f_3, f_4)$ and $\Diamond(p_1, p_2, p_3, p_4):=\Lambda_G(p_1, p_2, p_3, p_4).$

\item  If $G$ is the $P_3$, we denote
$\sqcap(f_1,f_2, f_3, f_4):=\Lambda_G(f_1, f_2, f_3, f_4)$ and 
$\sqcap(p_1, p_2, p_3, p_4):=\Lambda_G(p_1, p_2, p_3, p_4).$

\item If $G$ is the $K_3$ + tail (a kite), we denote
$\unlhd(f_1, f_2,f_3, f_4):=\Lambda_G(f_1, f_2, f_3, f_4)$ and 
$\unlhd(p_1, p_2, p_3, p_4):=\Lambda_G(p_1, p_2, p_3, p_4).$

\item If $G$ is the $Y$-shape (a space station), we denote
$Y(f_1, f_2,f_3, f_4):=\Lambda_G(f_1, f_2, f_3, f_4)$ and 
$Y(p_1, p_2, p_3, p_4):=\Lambda_G(p_1, p_2, p_3, p_4).$
\end{enumerate}
\end{itemize}

\begin{itemize}
\item We denote $\Lambda_G(p_1, \ldots, p_n)\lessapprox 1$ if the inequality (\ref{generalproblem}) holds true for all characteristic functions on $\mathbb F_q^d.$
\item By $\mathbb{F}_q^*$, we mean the set of all non-zero elements in $\mathbb{F}_q$. 
\item For $t\in \mathbb F_q^*$,   we denote by $S_t^{n-1}$ the sphere of radius $t$ centered at the origin in $\mathbb F_q^n:$
$$ S_t^{n-1}:=\{x\in \mathbb F_q^n: ||x||=t\}.$$
Unless otherwise specified in this paper, $d$ represents the general dimension of $\mathbb F_q^d, d\ge 2.$ 
When $n=d$, we write $S_t$ instead of $S_t^{d-1}$ for simplicity. 
\item We identify the set $S_t$ with its indicator function $1_{S_t}$, namely, $S_t(x)=1_{S_t}(x)$.
\item We write $\delta_0$ for the indicator function of the set of  the zero vector in $\mathbb F_q^d.$
\item For positive numbers $A, B>0,$  we write $A\lesssim B$ if $A\le C B$ for some constant $C>0$ independent of $q$, the size of the underlying finite field $\mathbb F_q.$  The notation $A\sim B$ means that $A\lesssim B$ and $B\lesssim A.$
\end{itemize}

We also study the boundedness relation between the operators associated with a graph $G$ and its subgraph $G'$ with $n$-vertices. Throughout the paper, we always assume that the graph $G$ and its subgraph $G'$ are connected ordered graphs with $|G|=|G'|$  in $\mathbb F_q^d,$ and two vertices $x,y$ in $G$ is connected if $||x-y||=t\ne 0.$

In Theorem \ref{mainthmS5}, we will see that  any exponents $1\le p_1, p_2, p_3 \le \infty$ with $\Delta(p_1, p_2, p_3)\lesssim 1$ satisfy that $\Lambda(p_1, p_2, p_3) \lesssim  1$. Notice that $P_2$ can be considered as a subgraph of $K_3,$ and  the operators $\Lambda$ and $\Delta$ are related to the graphs $P_2$ and $K_3,$ respectively. Hence, in view of Theorem \ref{mainthmS5}, one  may have a question that, ``Compared to a graph $G$,  does the operator associated with its subgraph yield less restricted mapping exponents?''. More precisely, one may pose the following question.

\begin{question} \label{PConj}
Suppose that $G'$ is a subgraph of the graph $G$ with $n$ vertices in $\mathbb F_q^d.$  Let  $1\le p_i\le \infty, 1\le i\le n.$ If $\Lambda_G(p_1, \ldots, p_n)\lesssim 1,$ is it true that $\Lambda_{G'}(p_1, \ldots, p_n) \lesssim 1?$
\end{question}

Somewhat surprisingly, the answer turns out to be no! When $G=K_3$ and $G'=P_2,$  the answer to Question \ref{PConj} is positive as Theorem \ref{mainthmS5} shows. However, it turns out that there exist a graph $G$ and its subgraph $G'$ yielding a negative answer, although the answers are positive for the most graphs which we consider in this paper. For example, the answer to Question \ref{PConj} is negative when $G$ is the $C_4$ + diagonal and $G'$ is the $C_4$ (see Theorem \ref{CExa}).

Since the general answer to Question \ref{PConj} is not always positive, we pose the following natural question.
\begin{problem}\label{P1K}
Find  general properties of the graph $G$ and its subgraph $G'$ which yield a positive answer to Question \ref{PConj}. 
\end{problem}
The main goal of this paper is to address a conjecture on this problem and to confirm it in two dimensions. To precisely state our conjecture on the problem, let us review the standard definition and notation for the minimal degree of a graph.
\begin{definition} 
The Minimum Degree of a graph $G$, denoted by $\delta(G)$, is defined as the degree of the vertex with the least number of edges incident to it.
\end{definition}

We propose the following conjecture which can be a solution of Problem \ref{P1K}.
\begin{conjecture} \label{PConjK} Let $G'$ be a subgraph of the graph $G$ in $\mathbb F_q^d, d\ge 2,$ with $n$ vertices, and let $1\le p_i\le \infty, 1\le i\le n.$ In addition, assume that 
\begin{equation}\label{mainConj} \min\left\{\delta(G), d\right\} > \delta(G').\end{equation}
Then if $\Lambda_G(p_1, \ldots, p_n)\lesssim 1$, we have  $\Lambda_{G'}(p_1, \ldots, p_n)\lesssim 1.$
%If the number of vertices with degree one in $G$ is less than that in $G',$ and  $\Lambda_G(p_1, \ldots, p_n)\lesssim 1$, then  $\Lambda_{G'}(p_1, \ldots, p_n)\lesssim 1.$
\end{conjecture}

Note that the condition \eqref{mainConj} in Conjecture \ref{PConjK} is equivalent to the following:
\begin{equation}\label{EquiCon} (i)~ \delta(G)>\delta(G') \quad \mbox{and} \quad (ii) ~d> \delta(G').\end{equation}
We have some comments and further questions below, regarding the above conjecture and our main theorems which we will state and prove in the body of this paper.
\begin{itemize}
\item 
Our results in this paper confirm Conjecture \ref{PConjK}, possibly up to endpoints, for all graphs $G$ and their subgraphs $G'$ on $n=3, 4$ vertices in $\mathbb F_q^2$ (see Theorem \ref{mainthm22}). Also note that when $d=2$ and $n=3, 4$, the condition \eqref{mainConj} is equivalent to the first condition $(i)$ in \eqref{EquiCon} since $\delta(G)\le 2$ (see the figures above).
\item It will be shown that the conclusion of Conjecture \ref{PConjK} cannot be reversed at least for $n=3, 4$ in two dimensions (see Remark \ref{Remark5.6} for $n=3$, and see Remarks \ref{remark8.10}, \ref{remark8.11K}, \ref{ReRk}, \ref{remark10.6K} for $n=4$.)
\end{itemize}

It is worth investigating whether the key hypothesis \eqref{mainConj} of Conjecture \ref{PConjK} can be relaxed.
\begin{itemize}
\item  The conclusion of Conjecture \ref{PConjK} does not hold in general if the $ >$ in the assumption \eqref{mainConj} is replaced by $\ge$. To see this,  consider  $G=C_4+$diagonal and $G'=C_4$ on 4 vertices in $\mathbb F_q^d,$ with $d=2.$ It is obvious that $G'$ is a subgraph of $G$ and $\min\{\delta(G), d\}=\delta(G')=2.$ However,  Theorem \ref{CExa} (ii) implies that the conclusion of Conjecture \ref{PConjK} is not true. 
\item We are not sure what can we say about the conclusion of Conjecture \ref{PConjK} if the main hypothesis \eqref{mainConj} of Conjecture \ref{PConjK} is relaxed by the second one of the conditions \eqref{EquiCon}. To be precise, when \eqref{mainConj} is replaced by the second statement of \eqref{EquiCon}, that is, $d> \delta(G'),$ we do not have a definitive answer even for $n=4$ in $\mathbb F_q^2.$ For instance, let $G=Y$-shape and $G'=K_3+$ a tail on $\mathbb F_q^2.$  Then it is clear that $d=2>\delta(G')=1$ and $1=\delta(G)=\delta(G')$ and so this provides an example that does not satisfy the assumption \eqref{mainConj} of Conjecture \ref{PConjK} but satisfy the second statement of \eqref{EquiCon}. Unfortunately,  in this paper we have not found any inclusive bounedness relations between the operators corresponding to such graphs.  In order to exclude this uncertain case, both conditions in \eqref{EquiCon} were taken as the hypothesis for Conjecture \ref{PConjK}, namely the condition \eqref{mainConj}.
\end{itemize}

The rest of this paper is organized as follows: In Section $2$, we recall known results on the spherical averaging operator, which functions as a fundamental tool to prove our theorems. Sections $3$ to $10$ are devoted to the presentation and proofs of our main results associated to the graphs mentioned above. The appendix contains some technical lemmas on the number of intersection points of two spheres in $\mathbb{F}_q^d$. 

\vskip.25in 

\section{The spherical averaging problem}
In the finite field setting, Carbery, Stones, and Wiright \cite{CSW08} initially formulated and studied the averaging problem over the varieties defined by vector valued polynomials. This problem for general varieties was studied by Chun-Yen Shen and the third listed author  \cite{KS13}. Here, we introduce the standard results on the averaging problem over the spheres. We adopt the notation in \cite{KS13}. 

Let $dx$ be the normalizing counting measure on $\mathbb F_q^d.$
For each non-zero $t$,  we endow the sphere $S_t$ with the normalizing surface measure $d\sigma_t.$ 
We recall that 
$$ d\sigma_t(x)= \frac{q^d}{|S_t|} 1_{S_t}(x) dx$$
so that we can identify the measure $d\sigma_t$ with the function $\frac{q^d}{|S_t|} 1_{S_t}$ on $\mathbb F_q^d.$

The spherical averaging operator  $A_{S_t}$ is defined by
\begin{equation}\label{DefA} A_{S_t}f(x)=f\ast d\sigma_t(x)=\int_{S_t} f(x-y) d\sigma_t(y) =\frac{1}{|S_t|} \sum_{y\in S_t} f(x-y),\end{equation}
where  $f$ is a function on $\mathbb F_q^d.$ By a change of variables, we also have
\begin{equation}\label{AAdef} A_{S_t}f(x)=\frac{1}{|S_t|} \sum_{y\in \mathbb F_q^d} S_t(x-y) f(y). \end{equation}

For $1\le p, r\le \infty,$  we define $A_{S_t}(p\to r)$ to be the smallest number such that the averaging estimate
\begin{equation}\label{AEK} ||f\ast d\sigma_t||_{L^r(\mathbb F_q^d, dx)} \le A_{S_t}(p\to r) ||f||_{L^p(\mathbb F_q^d, dx)}\end{equation}
holds for all test functions $f$ on $\mathbb F_q^d.$ 
\begin{problem} [Spherical averaging problem] Determine all exponents $1\le p, r\le \infty$ such that 
$$ A_{S_t}(p\to r)\lesssim 1.$$
\end{problem}
\begin{notation}
From now on, we simply write $A$ for the spherical averaging operator $A_{S_t}.$ 
\end{notation}

By testing  \eqref{AEK} with $f=\delta_0$ and by using the duality of the averaging operator, it is not hard to notice that 
the necessary conditions for the boundedness of $A(p\to r)$ are as follows:
$ (1/p, 1/r)$ is contained in the convex hull of points $(0,0), (0, 1),  (1,1), $ and $(\frac{d}{d+1},  \frac{1}{d+1} ).$

Using the Fourier decay estimate on $S_t$ and its cardinality,  it can be shown that  these necessary conditions are sufficient. For the reader's convenience, we give a detail proof although the argument is standard, as is well known in the literature such as \cite{CSW08, KS13}.

\begin{theorem}\label{sharpA}  Let $ 1\le p, r\le \infty$ be numbers such that $(1/p, 1/r)$ lies on the convex hull of points $(0,0), (0, 1),  (1,1), $ and $(\frac{d}{d+1},  \frac{1}{d+1} ).$ Then we have $ A(p\to r) \lesssim 1.$
\end{theorem}

\begin{proof}
Since both $d\sigma_t$ and $dx$ have total mass $1,$  it follows from Young's inequality for  convolution functions that  if $1\leq r \leq p\leq \infty,$ then
\begin{equation}\label{K1} \|f\ast d\sigma_t\|_{L^r({\mathbb F_{q}^d, dx})} \leq \|f\|_{L^p({\mathbb F_{q}^d, dx})}.\end{equation}
We notice that these results do not hold for the Euclidean Averaging problem.

By the interpolation and  the duality,  we only need to establish the following critical estimate:
$$ A\left(\frac{d+1}{d}\to d+1\right)\lesssim 1.$$
It is well known that for non-zero $t$,  
$$|(d\sigma_t)^\vee(m)|:=\left|\frac{1}{|S_t|} \sum_{x\in S_t} \chi(m\cdot x)\right|\lesssim q^{-\frac{(d-1)}{2}}\quad \mbox{for all}~~ m\neq (0,\dots,0),$$
where $\chi$ denotes a non-trivial additive character of $\mathbb F_q$ (see the proof of Lemma 2.2 in \cite{IR07}). 

Since $|S_t|\sim q^{d-1},$ we complete the proof by combining this Fourier decay estimate with the following lemma.
\begin{lemma}\label{Carbery}
Let $d\sigma$ be the normalized surface measure on an variety $S$ in $\mathbb F_q^d$  with $|S|\sim q^{d-1}.$
If $|(d\sigma)^\vee(m)|\lesssim  q^{-\frac{k}{2}}$ for all $m\in {\mathbb F_q^d}\setminus (0,\dots,0)$ and for some $k>0,$ then we have 
$$ A\left(\frac{k+2}{k+1}\to k+2\right)\lesssim 1.$$
\end{lemma}
\textbf{Proof of Lemma \ref{Carbery}} Define $K=(d\sigma)^\vee -\delta_0.$ 
We want to prove that for every function $f$ on $({\mathbb F_q^d},dx)$, 
$$\|f\ast d\sigma\|_{L^{k+2}({\mathbb F_q^d},dx)}\lesssim \|f\|_{L^{\frac{k+2}{k+1}}({\mathbb F_q^d},dx)},$$
where $dx$ is the normalized counting measure on $\mathbb F_q^d.$

Since $d\sigma= \widehat{K}+\widehat{\delta_0}= \widehat{K}+1$ and $\|f\ast 1\|_{L^{k+2}({\mathbb F_q^d},dx)}\lesssim \|f\|_{L^{\frac{k+2}{k+1}}({\mathbb F_q^d},dx)}$, it suffices to show that for every $f$ on $({\mathbb F_q^d},dx),$
\begin{equation}
\|f\ast \widehat{K}\|_{L^{k+2}({\mathbb F_q^d},dx)}\lesssim \|f\|_{L^{\frac{k+2}{k+1}}({\mathbb F_q^d},dx)}.
\end{equation}
Notice that this inequality can be obtained by interpolating the following two estimates:
\begin{equation}\label{first1}
\|f\ast \widehat{K}\|_{L^2({\mathbb F_q^d},dx)}\lesssim q^{-\frac{k}{2}}\|f\|_{L^2({\mathbb F_q^d},dx)},
\end{equation}
and 
\begin{equation}\label{second2}
\|f\ast \widehat{K}\|_{L^{\infty}({\mathbb F_q^d},dx)}\lesssim  q \|f\|_{L^1({\mathbb F_q^d},dx)}.
\end{equation}
It remains to prove the above two estimates.
The inequality (\ref{first1}) follows from  the Plancherel theorem, the size assumption of $|(d\sigma)^\vee|$, and the definition of $K.$
On the other hand, the inequality (\ref{second2}) follows from Young's inequality and the observation that $\|\widehat{K}\|_{L^\infty({\mathbb F_q^d},dx)}\lesssim  q.$  
\end{proof}

The boundary points of the convex hull play an important role in the application of Theorem \ref{sharpA}. More precisely, we will apply the following result.
\begin{lemma}\label{BALem}
Let $1\le p, r\le \infty$ and let $A$ denote the averaging operator over the sphere $S_t, t\ne 0,$ in $\mathbb F_q^d, d\ge 2.$
\begin{enumerate}
\item [(i)] If $1\le \frac{1}{p} \le \frac{d}{d+1}$ and $ \frac{1}{r}= \frac{1}{dp}$, then  $A(p\to r)\lesssim 1.$
\item [(ii)] If $ \frac{d}{d+1} \le  \frac{1}{p} \le 1$ and $ \frac{1}{r}=\frac{d}{p}-d+1$, then $A(p\to r)\lesssim 1.$
\end{enumerate}
\end{lemma}

\begin{proof} Let $L_1$ and $L_2$ denote the line segments connecting the two points $(0, 0),  (d/(d+1),  ~1/(d+1) )$ and 
 the two points $(d/(d+1),  ~1/(d+1) ),  (1,1),$ respectively. 
Theorem \ref{sharpA} implies that for any point $(1/p, 1/r)$ lying on either $L_1$ or $L_2,$ we have  $A(p\to r) \lesssim 1.$  Hence, the first part of the lemma  follows
since $L_1=\{(1/p, 1/r):  0\le 1/p \le d/(d+1),  ~ 1/r= 1/(dp)\}.$ 
Similarly,  since $L_2=\{(1/p, 1/r):  d/(d+1)\le 1/p\le 1, ~   1/r= d/p-d+1\},$  the second part of the lemma is obtained, as required.
\end{proof}

\section{Sharp mapping properties for the $K_2$-Operator }
In this section, we provide the sharp mapping properties of the operator associated to the graph $K_2.$  To this end, as described below, we relate the problem to the spherical averaging problem. 

As usual, the inner product of the non-negative real-valued functions $f, g$ on $\mathbb F_q^d$  is defined as
$$ <f, g>:=||f g||_1=\frac{1}{q^d} \sum_{x\in \mathbb F_q^d} f(x) g(x).$$
Let $t\in \mathbb F_q^*$ and let $f_1,f_2$ be non-negative real-valued functions on $\mathbb F_q^d, d\ge 2.$ Then the operator $L$ associated with the graph $K_2$ is defined as
\begin{equation}\label{lineDe} L(f_1,f_2)= \frac{1}{q^d|S_t|} \sum_{x^1, x^2\in \mathbb F_q^d} S_t(x^1-x^2) f_1(x^1) f_2(x^2),\end{equation}
and it is called the $K_2$-Operator on $\mathbb F_q^d.$ Here, the quantity $q^d|S_t|$ represents the normalizing factor $\mathcal{N}(G)$ in \eqref{generalform} when the graph $G$ is $K_2.$

By a change of variables, we can write
\begin{equation}\label{CVA} L(f_1,f_2) =  \frac{1}{q^d} \sum_{x^1\in \mathbb F_q^d} f_1(x^1) \left( \frac{1}{|S_t|} \sum_{x^2\in\mathbb F_q^d} f_2(x^1-x^2) S_t(x^2)\right)=<f_1, Af_2>,\end{equation}
where $A$ denotes the averaging operator related to the sphere $S_t.$ 
Likewise, we also obtain that $ L(f_1, f_2) = < Af_1,  f_2>.$

In a usual way, we define the operator norm of the $K_2$-Operator $L$ as follows.
\begin{definition} Let $1\le p_1, p_2\le \infty.$ 
We define $L(p_1, p_2)$ as the smallest constant such that   the following estimate holds for all functions $f_1, f_2$ on $\mathbb F_q^d:$
$$ L(f_1,f_2) \le L(p_1, p_2) ||f_1||_{p_1} ||f_2||_{p_2}.$$
\end{definition} 
The main goal of this section is to address all numbers $1\le p_1, p_2\le \infty$ satisfying  $L(p_1, p_2)\lesssim 1.$

We begin with the necessary conditions for the boundedness of the $K_2$-Operator $L$ on $\mathbb F_q^d.$
\begin{proposition}\label{NesLam}
Let $1\le p_1, p_2\le\infty.$ 
Suppose that $ L(p_1, p_2) \lesssim 1$. Then we have
$$\frac{1}{p_1} +  \frac{d}{p_2}\le d  \quad \mbox{and}\quad  \frac{d}{p_1} +  \frac{1}{p_2}\le d .$$
\end{proposition}

\begin{proof} By symmetry, it is clear that $ L(p_1, p_2)\lesssim 1 \iff  L(p_2, p_1)\lesssim 1.$
Hence, it suffices to prove the first listed conclusion that $\frac{1}{p_1} +  \frac{d}{p_2}\le d.$

 From \eqref{CVA} and our assumption that $ L(p_1, p_2) \lesssim 1,$ we must have
$$ L(f_1,f_2) :=  \frac{1}{q^d} \sum_{x^1\in \mathbb F_q^d} f_1(x^1) \left( \frac{1}{|S_t|} \sum_{x^2\in \mathbb F_q^d} f_2(x^1-x^2) S_t(x^2)\right) \lesssim ||f_1||_{p_1} ||f_2||_{p_2}.$$
 We test this inequality with $f_1=1_{S_t}$ and $f_2=\delta_0.$  Then
 $$ L(f_1, f_2)= \frac{1}{q^d} \sum_{x^1\in S_t} \frac{1}{|S_t|} =q^{-d},$$
 and
 $$ ||f_1||_{p_1} ||f_2||_{p_2} \sim  (q^{-d} |S_t|)^{1/p_1}(q^{-d} )^{1/p_2} \sim q^{-\frac{1}{p_1}-\frac{d}{p_2}}.$$
 By a direct comparison, we get the desired result.
 \end{proof}
 
 \begin{remark}  For $1\le p_1, p_2\le \infty, $ one can note that
 $\frac{1}{p_1} +  \frac{d}{p_2}\le d$ and  $\frac{d}{p_1} +  \frac{1}{p_2}\le d $ if and only if  $(1/p_1, 1/p_2)\in [0, 1]\times [0, 1]$ lies on the convex hull of points $(0, 0), (0, 1),  (\frac{d}{d+1},  \frac{d}{d+1}), (1,0).$
 \end{remark}
 
Let us move to the sufficient conditions on the exponents $1\le p_1, p_2\le \infty$ such that $L(p_1, p_2) \lesssim 1.$
We will invoke the following lemma which demonstrates that the boundedness of the $K_2$-Operator $L$ can be directly determined by the spherical averaging estimates over the finite fields.

For $1\le p\le \infty$,  we denote by $p'$ the H\"older conjugate of $p,$ namely,  $1/p+1/p'=1.$
\begin{lemma} \label{distancelem}
Suppose that  $A(p_2 \to p_1')\lesssim 1$ with $1\le p_1, p_2 \le \infty.$ Then we have
$$ L(p_1, p_2) \lesssim 1.$$
\end{lemma}
\begin{proof}  Since $L(f_1, f_2) = <f_1, Af_2>,$  it follows by H\"older's inequality that 
$$ L(f_1, f_2) \le ||f_1||_{p_1} ||Af||_{p_1'} \lesssim ||f_1||_{p_1} ||f_2||_{p_2},$$
where  the averaging assumption was used in the last inequality.
\end{proof}

We now show that the necessary conditions are in fact sufficient conditions for $L(p_1, p_2) \lesssim 1.$
\begin{theorem}[Sharp boundedness result for the $K_2$-Operator on $\mathbb F_q^d$] \label{GenRL}Let $1\le p_1, p_2\le \infty.$  Then we have
$$ L(p_1, p_2)\lesssim 1 \quad \mbox{if and only if}\quad \frac{1}{p_1} +  \frac{d}{p_2}\le d,  \quad  \frac{d}{p_1} +  \frac{1}{p_2}\le d .$$
\end{theorem}
\begin{proof}
By Proposition \ref{NesLam},  it will be enough to prove that $L(p_1, p_2)\lesssim 1$ for all $1\le p_1, p_2\le \infty$ satisfying  $$\frac{1}{p_1} +  \frac{d}{p_2}\le d,  \quad  \frac{d}{p_1} +  \frac{1}{p_2}\le d .$$
By the interpolation theorem and the nesting property of  the norm, it suffices to establish the estimates on  the critical end-points $(1/p_1, 1/p_2) \in [0, 1]\times [0,1]$, which are     
$ (0,1), (1,0),$ and $(d/(d+1), d/(d+1)).$ In other words,  it remains to prove the following estimates:
$$ L(\infty, 1)\lesssim 1,   ~L(1, \infty)\lesssim 1, ~ L\left(\frac{d+1}{d}, ~\frac{d+1}{d}\right)\lesssim 1.$$
Using Lemma \ref{distancelem},  matters are reduced to establishing the following averaging estimates:
$$ A(1\to 1)\lesssim 1,  ~~A(\infty\to \infty)\lesssim 1,  ~~ A\left(\frac{d+1}{d} \to d+1\right)\lesssim 1.$$
However, these averaging estimates  are clearly valid by Theorem \ref{sharpA}. Hence, the proof is complete. 
\end{proof}

The following result is a special case of Theorem \ref{GenRL}, but it is very useful in practice.
\begin{corollary} \label{LamRes} For any dimensions $d\ge 2, $ we have
$ L\left( \frac{d+1}{d},  \frac{d+1}{d}\right) \lesssim 1.$
\end{corollary}
\begin{proof}
Notice that if $p_1=p_2=\frac{d+1}{d}$, then it satisfies that $\frac{1}{p_1} +  \frac{d}{p_2}\le d$ and $ \frac{d}{p_1} +  \frac{1}{p_2}\le d.$ 
Hence, the statement follows immediately from Theorem \ref{GenRL}.
\end{proof}

\section{ Boundedness problem for the $K_3$-Operator}
 
Let $t\in \mathbb F_q^*.$  The operator $\Delta$ related to the graph $K_3$ can be defined as
\begin{equation} \label{KDeltatD} \Delta(f_1, f_2, f_3)= \frac{1}{q^d} \frac{1}{|S_t| |S_t^{d-2}|} \sum_{x^1, x^2, x^3\in \mathbb F_q^d} S_t(x^1-x^2) S_t(x^2-x^3) S_t(x^3-x^1) \prod_{i=1}^3 f_i(x^i) ,\end{equation}
where each $f_i, i=1,2,3,$ is a non-negative real-valued function on $\mathbb F_q^d,$ and the quantity $q^d|S_t| |S_t^{d-2}|$ stands for the normalizing factor $\mathcal{N}(G)$ in \eqref{generalform} when $G=K_3.$
We name the operator $\Delta$ as the $K_3$-Operator on $\mathbb F_q^d.$ 

\begin{definition} Let $1\le p_1, p_2, p_3\le \infty.$ 
We define $\Delta(p_1, p_2, p_3)$ as the best constant such that   the following estimate holds  for all non-negative real-valued functions $f_i, i=1,2,3,$ on $\mathbb F_q^d:$
$$ \Delta(f_1,f_2, f_3) \le \Delta(p_1, p_2, p_3) ||f_1||_{p_1} ||f_2||_{p_2} ||f_3||_{p_3}.$$
\end{definition} 

The purpose of this section is to find the numbers $1\le p_1, p_2, p_3\le \infty$ such that $\Delta(p_1, p_2, p_3)\lesssim 1.$ 
 
When the dimension $d$ is two, we will settle this problem up to the endpoint estimate. To this end, we relate our problem to the estimate of the Bilinear Averaging Operator (see \eqref{KdefB}) for which we establish the sharp bound.

On the other hand, as we shall see, in three and higher dimensions $d\ge 3,$ it is not easy to deduce the sharp results. However, when one of the exponents $p_1, p_2, p_3$ is $\infty$, we will able to obtain the optimal results. This will be done by applying Theorem \ref{GenRL}, the boundedness result for the $K_2$-Operator $L$ on $\mathbb F_q^d.$ 

We begin by deducing necessary conditions for our problem in $\mathbb F_q^d, d\ge 2.$ Recall that for $d=2,$ we pose an additional restriction that  $3\in \mathbb F_q$ is a square number. 
\begin{proposition} [Necessary conditions for the boundedness of $\Delta$]\label{NCD123}
Let $1\le p_1, p_2, p_3\le\infty.$ 
Suppose that $ \Delta(p_1, p_2, p_3) \lesssim 1$. Then we have
$$\frac{d}{p_1} +  \frac{1}{p_2}+ \frac{1}{p_3}\le d,  ~~\frac{1}{p_1} +  \frac{d}{p_2}+ \frac{1}{p_3}\le d,~~\frac{1}{p_1} +  \frac{1}{p_2}+ \frac{d}{p_3}\le d.$$
In particular, when $d = 2,$ it can be shown by Polymake\footnote{Polymake is software for the algorithmic treatment of convex polyhedra.}\cite{AGHJLP, GJ} that $(1/p_1, 1/p_2, 1/p_3)$ is contained in the convex hull of the points:  
$(0,0,1),$ $(0,1,0),$ $(2/3,2/3,0),$ $(1/2,1/2,1/2),$ $(2/3,0,2/3),$ $(1,0,0),$ $(0,0,0),$ $(0,2/3,2/3).$
\end{proposition}
\begin{proof} 
We only prove the first inequality in the conclusion since  we can establish other inequalities by symmetric property of $\Delta(f_1,f_2, f_3).$
We will use the simple fact  that $x\in S_t$ if and only if $-x\in S_t.$
In the definition \eqref{KDeltatD},  taking $f_1=\delta_0, f_2=1_{S_{t}},$ and $f_3=1_{S_{t}},$   we see that
$$ ||f_1||_{p_1} ||f_2||_{p_2} |||f_3||_{p_3} \sim q^{-\frac{d}{p_1}}  q^{-\frac{1}{p_2}} q^{-\frac{1}{p_3}},$$
%and
$$ \Delta(f_1,f_2,f_3) = \frac{1}{q^d}   \frac{1}{|S_{t}||S_t^{d-2}|} \left(\sum_{x^2, x^3\in S_{t}: ||x^2-x^3||=t} 1\right)\sim q^{-d},$$ 
where the last similarity above  follows from  Corollary \ref{KKo} in Appendix with our assumption that  $3\in \mathbb F_q$ is a square number for $d=2.$

By the direct comparison of these estimates, we obtain the required necessary condition. \end{proof}

\begin{remark} \label{mainRem}
In order to prove that  the necessary conditions in Proposition \ref{NCD123} are sufficient conditions for $d=2,$  we only need to establish the following critical endpoint estimates:
$\Delta( 2, 2, 2)\lesssim 1,$ $\Delta(\infty, \infty, \infty)\lesssim 1,$ $\Delta(1, \infty, \infty)\lesssim 1,$ $\Delta(\infty, 1, \infty)\lesssim 1,$  $\Delta(\infty, \infty, 1)\lesssim 1,$ 
$\Delta\left( \frac{3}{2},  \frac{3}{2}, \infty\right)\lesssim 1,$  $\Delta\left( \frac{3}{2},  \infty, \frac{3}{2} \right)\lesssim 1,$  $\Delta\left( \infty, \frac{3}{2},  \frac{3}{2}\right)\lesssim 1.$ In fact, this claim follows by interpolating the critical points  given in the second part of Proposition \ref{NCD123}.
\end{remark}

\subsection{Boundedness results for $\Delta$ on $\mathbb F_q^d$}

The graph $K_2$ can be obtained by removing  any one of three vertices in the graph $K_3.$  Therefore,  the boundedness of $L(p_1, p_2)$ can determine the boundedness of  $\Delta(p_1, p_2, \infty).$ More precisely, we have the following result.

\begin{lemma}\label{SimL} Let $1\le a, b\le \infty.$ If $L(a, b)\lesssim 1,$ then 
$$\Delta(a, b, \infty)\lesssim 1, ~~ \Delta(a, \infty, b)\lesssim 1,~~ \mbox{and}\quad \Delta(\infty, a, b)\lesssim 1.$$
\end{lemma}

\begin{proof} 
By symmetry, to complete the proof, it suffices to prove that  
$\Delta(a,b, \infty)\lesssim 1.$ 

Recall from \eqref{KDeltatD} and \eqref{lineDe} that
\begin{equation*}  \Delta(f_1, f_2, f_3)= \frac{1}{q^d} \frac{1}{|S_t| |S_t^{d-2}|} \sum_{x^1, x^2, x^3\in \mathbb F_q^d} S_t(x^1-x^2) S_t(x^2-x^3) S_t(x^3-x^1) \prod_{i=1}^3 f_i(x^i),\end{equation*}
and
\begin{equation*} L(f_1,f_2)= \frac{1}{q^d|S_t|} \sum_{x^1, x^2\in \mathbb F_q^d} S_t(x^1-x^2) f_1(x^1) f_2(x^2).\end{equation*}
Since $f_i, i=1,2,3,$ are non-negative real-number functions on $\mathbb F_q^d,$ we have
$$\Delta(f_1,f_2, f_3) \le  \frac{1}{|S_t^{d-2}|}\left(\max_{x^1, x^2\in \mathbb F_q^d: ||x^1-x^2||=t}  \sum_{x^3\in \mathbb F_q^d: ||x^2-x^3||=t =||x^3-x^1||} f_3(x^3)\right) L(f_1, f_2).$$
Since $|S_t^{d-2}|\sim q^{d-2}$ and $ L(f_1,f_2)\lesssim ||f_1||_{a} ||f_2||_b$ by our assumption,  it suffices to prove that the maximum value in the above parenthesis   is $\lesssim q^{d-2} ||f_3||_\infty.$
Let us denote by $I$ the maximum above. 

 By a change of variables, $x=x^1,  y=x^1-x^2,$  we see that
$$ I\le \left(\max_{x\in \mathbb F_q^d, y\in S_t}  \sum_{x^3\in \mathbb F_q^d: ||x-y-x^3||=t=||x^3-x||} 1 \right) ||f_3||_\infty. $$ 
By another change of variables by putting  $z=x-x^3$,   we get
$$ I\le\left( \max_{x\in \mathbb F_q^d, y\in S_t}  \sum_{z\in S_t: ||z-y||=t} 1\right) ||f_3||_\infty = \left( \max_{ y\in S_t}  \sum_{z\in S_t: ||z-y||=t} 1\right) ||f_3||_\infty.$$
Now applying Corollary \ref{KKo} in Appendix,  we conclude that $I\lesssim q^{d-2} ||f_3||_\infty$ as required.

\end{proof}

When one of $p_1, p_2, p_3$ is $\infty$,   we are able to obtain sharp boundedness results for $\Delta(p_1, p_2, p_3).$
\begin{theorem} \label{thmSi} Let $1\le a, b\le \infty$ satisfy that  
$\frac{1}{a} +  \frac{d}{b}\le d$ and  $\frac{d}{a} +  \frac{1}{b}\le d.$
Then we have $\Delta(a, b, \infty)\lesssim 1,$ $\Delta(a, \infty, b)\lesssim 1$ and $\Delta(\infty, a, b)\lesssim 1.$
\end{theorem}

\begin{proof}
 The statement follows immediately by combining Theorem \ref{GenRL}  and  Lemma \ref{SimL}.
\end{proof}

It is not hard to see from Proposition \ref{NCD123} that  the above theorem cannot be improved  in the case when  one of the exponents $p_1, p_2, p_3$ is $\infty.$ In particular,  we have the following critical endpoint estimates.

\begin{corollary} \label{mainC}We have
$\Delta\left(\frac{d+1}{d},  \frac{d+1}{d},  \infty\right) \lesssim 1,$ $\Delta\left(\frac{d+1}{d}, \infty ,\frac{d+1}{d} \right) \lesssim 1,$ and $\Delta\left(\infty, \frac{d+1}{d},  \frac{d+1}{d}\right) \lesssim 1.$
\end{corollary}
\begin{proof} 
Put $a=b=\frac{d+1}{d}.$ Then we see that $ \frac{1}{a} +  \frac{d}{b} =d = \frac{d}{a} +  \frac{1}{b},$
Hence, the statement of the corollary follows directly from Theorem \ref{thmSi}.
\end{proof}

Theorem \ref{thmSi} also implies the following result.
\begin{corollary}\label{mainCC}
We have
$\Delta(\infty, \infty, \infty)\lesssim 1,$ $\Delta(1, \infty, \infty)\lesssim 1,$ $\Delta(\infty, 1, \infty)\lesssim 1,$ $\Delta(\infty, \infty, 1)\lesssim 1.$
\end{corollary}
\begin{proof}
By taking $a=b=\infty$ in Theorem \ref{thmSi},  we obtain the estimate that  $\Delta(\infty, \infty, \infty)\lesssim 1.$
Now, by symmetry, it will be enough to show that $\Delta(1, \infty, \infty)\lesssim 1.$  However, this is easily shown by taking $a=1, b=\infty$ in Theorem \ref{thmSi}. Thus, the proof is complete.
\end{proof}

\begin{remark}\label{mainRRem} From Corollary \ref{mainC}, Corollary \ref{mainCC}, and Remark \ref{mainRem},  we  see that  to completely solve the problem on the boundedness of $\Delta(p_1, p_2, p_3)$ for $d=2,$ we only need to establish the following critical endpoint estimate
$$ \Delta\left( \frac{d+2}{d},  \frac{d+2}{d}, \frac{d+2}{d}\right)=\Delta(2,2,2)\lesssim 1.$$
\end{remark}

\subsection{Sharp restricted strong-type estimates in two dimensions}
Although Theorem \ref{thmSi} is valid for all dimensions $d\ge 2,$   it is not sharp, compared to the necessary conditions given in Proposition \ref{NCD123}. In this subsection, we will deduce  the sharp boundedness results up to the endpoints  for $\Delta(p_1, p_2, p_3)$ in two dimensions. To this end, we need the following theorem. 
\begin{theorem} \label{mainRSE} Let $\Delta$ be the $K_3$-Operator on $\mathbb F_q^2.$ Then, for all subsets $E, F, H $ of $\mathbb F_q^2,$  the following estimate holds:
$\Delta(E, F, H) \lesssim  ||E||_2 ||F||_2 ||H||_2.$
\end{theorem}

For $1\le p_1, p_2, p_3\le \infty$,  we say that the restricted strong-type $\Delta(p_1, p_2, p_3)$ estimate  holds if
 the estimate 
 $$\Delta(E, F, H) \lesssim  ||E||_{p_1} ||F||_{p_2} ||H||_{p_3}$$
 is valid for all subsets $E, F, H$ of $\mathbb F_q^2.$
 In this case, we write $ \Delta(p_1, p_2, p_3) \lessapprox 1.$
 %\subsection{Proof of Theorem \ref{mainRSE}}
 \begin{proof} The proof proceeds with some reduction. 
When $d=2$,  by a change of variables by letting $x=x^3,  y=x^3-x^1,  z=x^3-x^2,$  \eqref{KDeltatD}  becomes
\begin{equation*}\label{DeltatD}  \Delta(f_1, f_2, f_3)= \frac{1}{q^2}  \sum_{x\in \mathbb F_q^2} f_3(x) \left[ \frac{1}{|S_t|} \sum_{y, z\in \mathbb F_q^2} S_t(z-y) S_t(z)  S_t(y)  f_1(x-y) f_2(x-z) \right]. \end{equation*}

We define  $B(f_1, f_2)(x)$ as the value in the bracket above, namely,
\begin{equation}\label{KdefB}B(f_1, f_2)(x):=\frac{1}{|S_t|} \sum_{y, z\in S_t: ||z-y||=t}  f_1(x-y) f_2(x-z) .\end{equation}
We refer to this operator $B$  as ``the bilinear averaging operator''.
It is clear that
$$ \Delta(f_1,f_2,f_3)= \frac{1}{q^2} \sum_{x\in \mathbb F_q^2} B(f_1, f_2)(x)  f_3(x)=< B(f_1, f_2),~ f_3>.$$ 
By H\"older's inequality,  we have
$$\Delta(f_1,f_2,f_3) \le  ||B(f_1, f_2)||_2  ||f_3||_{2}.$$
Thus,  Theorem \ref{mainRSE} follows immediately from the reduction lemma below
\end{proof}
\begin{lemma} \label{mainlemK}Let $B(f_1, f_2)$ be the bilinear averaging operator defined as in \eqref{KdefB}. 
Then, for all subsets $E, F$ of $\mathbb F_q^2,$  we have
$$ ||B(E, F)||_2 \lesssim  ||E||_2  ||F||_2.$$
\end{lemma}
\begin{proof}
We begin by representing  the bilinear averaging operator $B(f_1,f_2).$ 
From \eqref{KdefB}, note that 
$$B(f_1, f_2)(x) =  \frac{1}{|S_{t}|} \sum_{y\in S_{t}}  f_1(x-y)\left( \sum_{z\in S_{t}: ||z-y||=t} f_2(x-z)\right).$$

For  each $y\in S_{t}$, let 
$ \Theta(y):= \{z\in S_t: ||z-y||=t\}.$
With this notation,   the bilinear averaging operator is written as
\begin{equation*} B(f_1, f_2)(x) =\frac{1}{|S_{t}|} \sum_{y\in S_{t}}f_1(x-y) \left( \sum_{z\in  \Theta(y)}  f_2(x-z)\right).\end{equation*}
Let $\eta$ denote the quadratic character of $\mathbb F_q^*.$ Recall that  $\eta(s)=1$ for a square number $s$ in $\mathbb F_q^*$, and $\eta(s)=-1$ otherwise.
Notice from Corollary \ref{KKo} in Appendix that  
$\Theta(y)$ is the empty set for all $y\in S_t$  if $d=2$ and $\eta(3)=-1.$  
In this case,  the problem is trivial since  $B(f_1, f_2)(x)=0$ for all $x\in \mathbb F_q^2.$  Therefore, when $d=2$, we always assume that $\eta(3)=1.$ 

Notice that   $|\Theta(y)|=2$ for all $y\in S_t,$ which follows from the last statement of  Corollary \ref{KKo} in Appendix. More precisely,   for each $y\in S_t,$ we can write
$$ \Theta(y)=\{ \theta y, ~ \theta^{-1} y\},$$
where $\theta y$ denotes the rotation of $y$ by  ``60 degrees,''  and  $||y-\theta y||=t =||y-\theta^{-1} y||.$

 From these observations, the bilinear averaging operator $B(f_1, f_2)$ can be represented as follows:
 \begin{equation}\label{DecB} B(f_1, f_2)(x) =B_\theta(f_1, f_2)(x) + B_{\theta^{-1}}(f_1, f_2)(x).\end{equation}
 Here, we define 
 \begin{equation}\label{defBthe} B_\theta(f_1, f_2)(x):= \frac{1}{|S_{t}|} \sum_{y\in S_{t}}f_1(x-y)  f_2(x- \theta y),\end{equation}
 and 
 $$ B_{\theta^{-1}}(f_1, f_2):= \frac{1}{|S_{t}|} \sum_{y\in S_{t}}f_1(x-y)  f_2(x- \theta^{-1} y).$$

In order to complete the proof of the lemma,  it suffices to establish the following two estimates: for all subsets $E, F$ of $\mathbb F_q^d,$ 
\begin{equation} \label{eqB1}
 ||B_\theta(E, F)||_2 \lesssim ||E||_2 ||F||_2,
\end{equation}
and 
\begin{equation}\label{eqB2}
 ||B_{\theta^{-1}}(E, F)||_2 \lesssim ||E||_2 ||F||_2.
 \end{equation}
 
 We will only provide the proof of the estimate \eqref{eqB1} since the proof of \eqref{eqB2} is the same.
 
Now we start proving the estimate \eqref{eqB1}.  Since $||E||_2^2=q^{-2}|E|$ and $ ||F||_2^2=q^{-2} |F|,$  it is enough to prove that
\begin{equation}\label{mimi} || B_\theta (E,  F)||_2^2 \lesssim q^{-4} |E||F|.\end{equation}
Without loss of generality, we may assume that $|E|\le |F|.$
By the definition, it follows that
$$ || B_\theta (E, F)||_2^2 =q^{-2} |S_t|^{-2}  \sum_{x\in \mathbb F_q^2} \sum_{y, y'\in S_t} E(x-y)E(x-y') F(x-\theta y) F(x-\theta y') = I + II,$$
where  the first term $I$ is the value corresponding to the case where $y=y'$, whereas the second term $II$ is corresponding to the case where $y\ne y'.$ We have
$$ I= q^{-2} |S_t|^{-2}  \sum_{y\in S_t}\sum_{x\in \mathbb F_q^2}  E(x-y)F(x-\theta y) .$$
Applying a change of variables by replacing $x$ with $x+y$,   we see that
$$ I=q^{-2} |S_t|^{-2}  \sum_{y\in S_t}\sum_{x\in \mathbb F_q^2}  E(x)F(x+ y-\theta y)
= q^{-2} |S_t|^{-2} \sum_{x\in E} \left( \sum_{y\in S_t} F(x+ y-\theta y)\right).$$
Observe that  $y-\theta y \ne y'-\theta y'$ for all $y, y'$ in $S_t$ with $y\ne y'.$  Then we see that
the value in the parentheses above is bounded above by  $|S_t\cap F|\le |F|.$ Therefore, we obtain the desired estimate:
$$ I \le q^{-2} |S_t|^{-2} |E||F|\sim q^{-4} |E||F|.$$
Next, it remains to show that
$ II \lesssim q^{-4} |E||F|.$
Since  we have assumed that $|E| \le |F|,$  it suffices to show that
$ II\lesssim q^{-4} |E|^2.$

By the definition of $II$,  it follows that
$$ II= q^{-2} |S_t|^{-2}  \sum_{x\in \mathbb F_q^2} \sum_{y, y'\in S_t: y\ne y'} E(x-y) E(x-y') F(x-\theta y) F(x-\theta y').$$
It is obvious that
$$ II \le q^{-2} |S_t|^{-2}  \sum_{y, y'\in S_t: y\ne y'} \sum_{x\in \mathbb F_q^2} E(x-y) E(x-y').$$
We use a change of variables by replacing $x$ with $x+y.$  Then we have
\begin{align*} II &\le q^{-2} |S_t|^{-2}  \sum_{x\in \mathbb F_q^2} E(x)  \left( \sum_{y, y'\in S_t: y\ne y'}  E(x+y-y')\right) \\
&= q^{-2} |S_t|^{-2}  \sum_{x\in \mathbb F_q^2} E(x) \left( \sum_{\mathbf{0} \ne u\in \mathbb F_q^2}  E(x+u) W(u)\right),\end{align*}
where  $W(u)$ denotes the number of pairs $(y, y') \in S_t \times S_t$ such that $u=y-y'$ and $y\ne y'.$
It is not hard to see that  for any non-zero vector $u \in \mathbb F_q^2$,    we have $W(u)\le 2.$   So we obtain that
$$ II \lesssim q^{-2} |S_t|^{-2}  \sum_{x\in \mathbb F_q^2} E(x) \left( \sum_{\mathbf{0} \ne u\in \mathbb F_q^2}  E(x+u) \right) \lesssim q^{-4} |E|^2 , $$
as required.
\end{proof}

In two dimensions,  we are able to obtain the optimal boundedness of $\Delta(p_1, p_2, p_3)$ except for one endpoint. Indeed, we have the following result.
\begin{theorem}\label{Best2D}  Let $1\le p_1, p_2, p_3\le \infty$ and  let $\Delta$ be the $K_3$-Operator on $\mathbb F_q^2.$ 
\begin{itemize}
\item [(i)] If  $\Delta(p_1, p_2, p_3) \lesssim 1$, then   
\begin{equation}\label{eq3C}\frac{2}{p_1} +  \frac{1}{p_2}+ \frac{1}{p_3}\le 2,  ~~\frac{1}{p_1} +  \frac{2}{p_2}+ \frac{1}{p_3}\le 2,~~\frac{1}{p_1} +  \frac{1}{p_2}+ \frac{2}{p_3}\le 2.\end{equation}

\item [(ii)] Conversely,  if $(p_1, p_2, p_3)$ satisfies all three inequalities \eqref{eq3C}, then  $\Delta(p_1, p_2, p_3) \lesssim 1 $  for  $(p_1, p_2, p_3)\ne (2,2,2),$  and   we have $\Delta(2,2,2)\lessapprox 1.$ 
\end{itemize}
\end{theorem}

\begin{proof}
 The first part  of the theorem is the special case of Proposition \ref{NCD123} with $d=2.$
 Now we  prove the second part. As stated in Proposition \ref{NCD123},  one  can notice by using Polymake \cite{AGHJLP, GJ} that  all the points $\left(\frac{1}{p_1},  \frac{1}{p_2}, \frac{1}{p_3}\right)\in [0, 1]^3$  satisfying all three inequalities \eqref{eq3C} are contained in the convex hull of the critical points  $$(1/2,1/ 2, 1/ 2), ~ (0, 0, 0),~ (1, 0, 0), ~ (0, 1,0),  ~(0, 0, 1), 
 \left( \frac{2}{3},  \frac{2}{3}, 0\right),~  \left( \frac{2}{3},  0, \frac{2}{3} \right), ~ \left( 0, \frac{2}{3},  \frac{2}{3}\right). $$
 Notice from Theorem \ref{mainRSE} with $d=2$ that the restricted strong-type estimate for the operator $\Delta$ holds for the point $(1/p_1, 1/p_2, 1/p_3)=(1/2, 1/2, 1/2).$  In addition, notice from Corollary \ref{mainC} and Corollary \ref{mainCC} with $d=2$ that 
 $\Delta(p_1, p_2, p_3)\lesssim 1$ for the above critical points $(1/p_1, 1/p_2, 1/p_3)$ except for $(1/2, 1/2, 1/2).$  Hence,
 the statement of the second part follows immediately by invoking the interpolation theorem.
 \end{proof}

\section{Boundedness results for the $P_2$-Operator}
For $t\in \mathbb F_q^*$ and  functions $f_i, i=1,2,3,$ on $\mathbb F_q^d,$   the operator $\Lambda$ associated with the graph $P_2$, called the $P_2$-Operator on $\mathbb F_q^d,$ is defined as
\begin{equation}\label{DefHinges}\Lambda(f_1,f_2,f_3)= \frac{1}{q^d|S_t|^2} \sum_{x^1, x^2, x^3\in \mathbb F_q^d} S_t(x^1-x^2)  S_t(x^2-x^3) f_1(x^1) f_2(x^2)f_3(x^3),\end{equation}
where  the quantity $q^d|S_t|^2$ stands for the normalizing factor $\mathcal{N}(G)$ in \eqref{generalform} when $G=P_2$. Note that
this can be written as
\begin{equation}\label{DDAver} \Lambda(f_1,f_2,f_3)= <f_2,~  Af_1 \cdot Af_3>.\end{equation}
\begin{definition} Let $1\le p_1, p_2, p_3\le \infty.$ 
We define $\Lambda(p_1, p_2, p_3)$ as the smallest constant such that   the following estimate holds for all non-negative real-valued functions $f_1, f_2, f_3$ on $\mathbb F_q^d:$
$$ \Lambda(f_1,f_2, f_3) \le \Lambda(p_1, p_2, p_3) ||f_1||_{p_1} ||f_2||_{p_2} ||f_3||_{p_3}.$$
\end{definition} 

In this section, we study the problem determining all numbers $1\le p_1, p_2, p_3\le \infty$ satisfying  $\Lambda(p_1, p_2, p_3)\lesssim 1.$ Compared to the $K_3$-Operator $\Delta$, this problem is much hard to find the optimal answers. Based on the formula \eqref{DDAver} with the averaging estimates in Lemma \ref{BALem},  we are able to address partial results on this problem (see Theorem \ref{SharpNS}).

%In this paper we will not pursue this problem but hope to address its answer in the future.

\begin{proposition} [Necessary conditions for the boundedness of $\Lambda$]\label{NecLam}
Let $1\le p_1, p_2, p_3\le\infty.$ 
Suppose that $ \Lambda(p_1, p_2, p_3) \lesssim 1$. Then we have
$$\frac{1}{p_1} +  \frac{d}{p_2}+ \frac{1}{p_3}\le d, \quad \frac{d}{p_1} +  \frac{1}{p_2}\le d, \quad \frac{1}{p_2} +  \frac{d}{p_3}\le d,  \quad \frac{d}{p_1} +  \frac{1}{p_2}+ \frac{d}{p_3}\le 2d-1.$$

Also, under this assumption when $d = 2$, it can be shown by  Polymake \cite{AGHJLP, GJ} that $(1/p_1, 1/p_2, 1/p_3)$ is contained in the convex hull of points: $(0,1,0),$ $(1/2,0,1),$ $(1,0,1/2),$$(1,0,0),$ $(5/6,1/3,1/2),$ $(1/2,1/3,5/6),$ $(2/3,2/3,0),$ $(0,2/3,2/3),$ $(0,0,0),$ $(0,0,1).$
\end{proposition}
\begin{proof}
Suppose that $\Lambda(p_1, p_2, p_3) \lesssim 1.$
Then, for all  functions $f_i, i=1,2,3,$ on $\mathbb F_q^d,$   we have
$$\Lambda(f_1,f_2,f_3)= \frac{1}{q^d|S_t|^2} \sum_{x^1, x^2, x^3\in \mathbb F_q^d} S_t(x^1-x^2)  S_t(x^2-x^3) f_1(x^1) f_2(x^2)f_3(x^3) \lesssim ||f_1||_{p_1}||f_2||_{p_2} ||f_3||_{p_3}.$$
We test the above inequality with  $f_2=\delta_0, f_1=f_3=1_{S_t}.$  It is plain to note that
$$ ||f_1||_{p_1}||f_2||_{p_2} ||f_3||_{p_3} = (q^{-d} |S_t|)^{\frac{1}{p_1}} q^{-\frac{d}{p_2}} (q^{-d} |S_t|)^{\frac{1}{p_3}} \sim  q^{-\frac{1}{p_1}-\frac{d}{p_2} -\frac{1}{p_3}},$$ 
and $$ \Lambda(f_1, f_2, f_3) = q^{-d}.$$
Therefore, we obtain that $q^{-d}\lesssim q^{-\frac{1}{p_1}-\frac{d}{p_2} -\frac{1}{p_3}}.$ This implies the first inequality in the conclusion that $ \frac{1}{p_1} +  \frac{d}{p_2}+ \frac{1}{p_3}\le d.$

To obtain the second inequality in the conclusion, we choose $f_1=\delta_0, f_2=1_{S_t},$ and $f_3=1_{\mathbb F_q^d}.$ Then it is easy to check that
\begin{equation*}\label{comPK} ||f_1||_{p_1}||f_2||_{p_2} ||f_3||_{p_3} \sim  q^{-\frac{d}{p_1}-\frac{1}{p_2} },\end{equation*}
and
$$ \Lambda(f_1, f_2, f_3) = \frac{1}{q^d |S_t|^2} \sum_{x^2\in S_t} \sum_{x^3\in \mathbb F_q^d}  S_t(x^2-x^3) 
=q^{-d}.$$
 Comparing these estimates gives the second inequality in the conclusion.

The third inequality in the conclusion can be easily obtained by switching the roles of $f_1$ and $f_3$ in the proof of the second one.

To deduce the last inequality in  the conclusion,  we  take $f_1=f_3=\delta_0$ and $f_2=1_{S_t}.$  Then
$$ ||f_1||_{p_1}||f_2||_{p_2} ||f_3||_{p_3}\sim q^{-\frac{d}{p_1}-\frac{1}{p_2} -\frac{d}{p_3}},$$ 
and $$ \Lambda(f_1, f_2, f_3) = \frac{1}{q^d|S_t|}\sim q^{-2d+1}.$$
From these, we have the required result that
$ \frac{d}{p_1} +  \frac{1}{p_2}+ \frac{d}{p_3}\le 2d-1.$
\end{proof}

By symmetry, it is not hard to note that $ \Lambda(p_1, p_2, p_3)\lesssim 1 \iff  \Lambda(p_3, p_2, p_1)\lesssim 1.$
In the following lemma, we prove that the boundedness question for the $P_2$-Operator $\Lambda$ is closely related to the spherical averaging problem over finite fields.  
\begin{lemma} \label{Hinge}
Suppose that  $\frac{1}{r_1}+ \frac{1}{p_2}+\frac{1}{r_3}=1,  A(p_1 \to r_1)\lesssim 1,$ and $A(p_3 \to r_3)\lesssim 1$ for some $1\le p_1, p_2, p_3, r_1, r_3 \le \infty.$ Then we have 
$ \Lambda(p_1, p_2, p_3) \lesssim 1.$
\end{lemma}

\begin{proof}  Since $ \Lambda(f_1,f_2,f_3)= <f_2, ~ Af_1 \cdot Af_3>,$  we obtain by H\"older's inequality with the first assumption that 
$$ \Lambda(f_1, f_2, f_3) \le ||Af_1||_{r_1} ||f_2||_{p_2}  ||Af_3||_{r_3} \lesssim ||f_1||_{p_1} ||f_2||_{p_2} ||f_3||_{p_3},$$
where  the averaging assumption was used for the last inequality. Hence,   $\Lambda(p_1, p_2, p_3) \lesssim 1,$   
as required.
\end{proof}

Now we state and prove  our boundedness results of $\Lambda(p_1, p_2, p_3)$ on $\mathbb F_q^d.$
\begin{theorem}\label{SharpNS} 
Let $1\le p_1, p_2, p_3\le\infty.$  Then, for the $P_2$-Operator $\Lambda$ on $\mathbb F_q^d,$ the following four statements hold: 
\begin{enumerate} 
\item [(i)]  If $0\le \frac{1}{p_1}, \frac{1}{p_3} \le \frac{d}{d+1}$ and $ \frac{1}{p_1}+\frac{d}{p_2}+ \frac{1}{p_3} \le d$, then $\Lambda(p_1, p_2, p_3)\lesssim 1. $

\item [(ii)] If $0\le \frac{1}{p_1} \le  \frac{d}{d+1} \le \frac{1}{p_3} \le 1$ and $ \frac{1}{dp_1}+\frac{1}{p_2}+ \frac{d}{p_3} \le d$, then $\Lambda(p_1, p_2, p_3)\lesssim 1. $

\item [(iii)] If $0\le \frac{1}{p_3} \le  \frac{d}{d+1} \le \frac{1}{p_1} \le 1$ and $ \frac{d}{p_1}+\frac{1}{p_2}+ \frac{1}{dp_3} \le d$, then $\Lambda(p_1, p_2, p_3)\lesssim 1. $

\item [(iv)] If $\frac{d}{d+1}\le \frac{1}{p_1}, \frac{1}{p_3} \le 1$ and $ \frac{d}{p_1}+\frac{1}{p_2}+ \frac{d}{p_3} \le 2d-1$, then $\Lambda(p_1, p_2, p_3)\lesssim 1. $
\end{enumerate}
\end{theorem}
\begin{proof} 
We proceed as follows.
\begin{enumerate}
 \item [(i)] By the nesting property of the norm,  it suffices to prove  it  in the case when
$0\le \frac{1}{p_1}, \frac{1}{p_3} \le \frac{d}{d+1}$ and $ \frac{1}{p_1}+\frac{d}{p_2}+ \frac{1}{p_3} =d.$
This equation can be rewritten as
$ \frac{1}{dp_1} +  \frac{1}{p_2}+ \frac{1}{dp_3}=1.$ 
Since $0\le \frac{1}{p_1}, \frac{1}{p_3}\le \frac{d}{d+1},$  we see from Lemma \ref{BALem} (i) that letting $ \frac{1}{r_1}=\frac{1}{dp_1}, \frac{1}{r_3}=\frac{1}{dp_3},$ we have  $A(p_1\to r_1)\lesssim 1$ and $A(p_3\to r_3)\lesssim 1.$ 

Since $\frac{1}{r_1} +\frac{1}{p_2}+\frac{1}{r_3}=1, $   applying Lemma \ref{Hinge} gives the required result.

\item [(ii)] As in the proof of the first part of the theorem, it will be enough to prove $\Lambda(p_1, p_2, p_3)\lesssim 1$ in the case when $0\le \frac{1}{p_1} \le  \frac{d}{d+1} \le \frac{1}{p_3} \le 1$ and $ \frac{1}{dp_1}+\frac{1}{p_2}+ \frac{d}{p_3} =d.$ Let $ \frac{1}{r_1}= \frac{1}{dp_1}$ and $\frac{1}{r_3}= \frac{d}{p_3}-d+1.$  Then by Lemma \ref{BALem} it follows that $A(p_1\to r_1)\lesssim 1$ and $A(p_3\to r_3)\lesssim 1.$ Also notice that $ \frac{1}{r_1}+ \frac{1}{p_2} + \frac{1}{r_3}=1.$ Hence,  Theorem \ref{SharpNS} (ii) follows from  Lemma \ref{Hinge}.
\item [(iii)] Switching  the roles of $p_1, p_2$,  the proof is exactly the same as that of the second part of this theorem. 
\item [(iv)] As before, it suffices to prove  the case when 
$\frac{d}{d+1}\le \frac{1}{p_1}, \frac{1}{p_3} \le 1$ and $ \frac{d}{p_1}+\frac{1}{p_2}+ \frac{d}{p_3} = 2d-1.$
Put $\frac{1}{r_k}=\frac{d}{p_k}-d+1$ for $k=1, 3.$  Then we see from  Lemma \ref{BALem} (ii) that  $A(p_k\to r_k)\lesssim 1$ for  $k=1, 3.$ Notice that $\frac{1}{r_1}+ \frac{1}{p_2} + \frac{1}{r_3}=1.$ Therefore,  using Lemma \ref{Hinge}  we finish the proof.
\end{enumerate}
\end{proof}

As a special case of Theorem \ref{SharpNS}, we obtain the following.
\begin{corollary} \label{HingeRes} For any dimensions $d\ge 2, $ we have
$\Lambda\left(\frac{d+1}{d},  \frac{d+1}{d-1},  \frac{d+1}{d}\right) \lesssim 1.$
\end{corollary}
\begin{proof}
This clearly follows from Theorem \ref{SharpNS}  by taking $p_1=p_3=\frac{d+1}{d},$  and $ p_2=\frac{d+1}{d-1}.$
\end{proof}

%Theorem \ref{SharpNS} does not cover all the numbers for the 
%necessary conditions in Proposition \ref{NecLam} for $\Lambda(p_1, p_2, p_3, p_4) \lesssim 1.$ However, it does not mean that Theorem \ref{SharpNS} is not sharp since it may be possible that the necessary conditions could be improved. 

While we do not know whether Theorem \ref{SharpNS} is optimal or not, the result will play a crucial role in proving the following theorem which implies that Conjecture \ref{PConjK} is true for the graph $K_3$ and its subgraph $P_2$ in all dimensions $d\ge 2$ (see Corollary \ref{CormainthmS5} below).

\begin{theorem} \label{mainthmS5} Let $\Delta$ and $\Lambda$ be the operators associated with $K_3$ and $P_2$, respectively, on $\mathbb F_q^d.$ 
Then if  $\Delta(p_1, p_2, p_3)\lesssim 1$ for $1\le p_1, p_2, p_3\le \infty,$ we have $\Lambda(p_1, p_2, p_3)\lesssim 1.$
\end{theorem}
\begin{proof}
Suppose that $\Delta(p_1, p_2, p_3)\lesssim 1$ for $1\le p_1, p_2, p_3\le \infty.$ Then, by Proposition \ref{NCD123}, the exponents $p_1, p_2, p_3$ satisfy the following three inequalities:
\begin{equation}\label{eqDelta}\frac{d}{p_1} +  \frac{1}{p_2}+ \frac{1}{p_3}\le d,  ~~\frac{1}{p_1} +  \frac{d}{p_2}+ \frac{1}{p_3}\le d,~~\frac{1}{p_1} +  \frac{1}{p_2}+ \frac{d}{p_3}\le d.\end{equation}
To complete the proof,  it remains to show that $\Lambda(p_1, p_2, p_3)\lesssim 1.$ We will prove this by considering the four cases depending on the sizes of $p_1$ and $p_3.$

\textbf{Case 1:} Suppose that $0\le \frac{1}{p_1}, \frac{1}{p_3} \le \frac{d}{d+1}.$ The condition \eqref{eqDelta} clearly implies that $\frac{1}{p_1}+\frac{d}{p_2}+ \frac{1}{p_3} \le d.$ 
Thus, by Theorem \ref{SharpNS} (i), we obtain the required conclusion that $\Lambda(p_1, p_2, p_3)\lesssim 1.$

\textbf{Case 2:} Suppose that $0\le \frac{1}{p_1} \le  \frac{d}{d+1} \le \frac{1}{p_3} \le 1.$ By  Theorem \ref{SharpNS} (ii), to prove that $\Lambda(p_1, p_2, p_3)\lesssim 1,$   it will be enough to show that
$$\frac{1}{dp_1}+\frac{1}{p_2}+ \frac{d}{p_3} \le d.$$
 However, this inequality clearly follows from the third inequality in \eqref{eqDelta} since $d\ge 2.$

\textbf{Case 3:} Suppose that $0\le \frac{1}{p_3} \le  \frac{d}{d+1} \le \frac{1}{p_1} \le 1.$ By  Theorem \ref{SharpNS} (iii), it suffices to show that 
$\frac{d}{p_1}+\frac{1}{p_2}+ \frac{1}{dp_3} \le d.$ However, this inequality can be easily obtained from the first inequality in \eqref{eqDelta}.

\textbf{Case 4:} Suppose that $\frac{d}{d+1}\le \frac{1}{p_1}, \frac{1}{p_3} \le 1.$ By Theorem \ref{SharpNS} (iv),  to show that $\Lambda(p_1, p_2, p_3)\lesssim 1$, we only need to prove that
$$\frac{d}{p_1}+\frac{1}{p_2}+ \frac{d}{p_3} \le 2d-1.$$
However, this inequality can be easily proven as follows:
$$ \frac{d}{p_1}+\frac{1}{p_2}+ \frac{d}{p_3} = \left(\frac{d}{p_1}+\frac{1}{p_2}+ \frac{1}{p_3}\right) + \frac{d-1}{p_3} \le d + \frac{d-1}{p_3} \le 2d-1,$$
where the first inequality follows from the first inequality in \eqref{eqDelta}, and the last inequality follows from a simple fact that $1\le p_3\le \infty.$
\end{proof}

\begin{remark}\label{Remark5.6} The reverse statement of Theorem \ref{mainthmS5} cannot be true. Indeed, we know by Corollary \ref{HingeRes} that $\Lambda\left(\frac{d+1}{d},  \frac{d+1}{d-1},  \frac{d+1}{d}\right) \lesssim 1.$  However, $\Delta\left(\frac{d+1}{d},  \frac{d+1}{d-1},  \frac{d+1}{d}\right)$ cannot be bounded, which can be easily shown by considering Proposition \ref{NCD123}, namely, the necessary conditions for the boundedness of $\Delta(p_1,p_2, p_3).$
\end{remark}

We invoke Theorem \ref{mainthmS5} to deduce the following result.
\begin{corollary} \label{CormainthmS5} Conjecture \ref{PConjK} is true for the graph $K_3$ and its subgraph $P_2$ in $\mathbb F_q^d$, $d\ge 2.$ 
\end{corollary}
\begin{proof} It is clear that $P_2$ is a subgraph of $K_3$ in $\mathbb F_q^d.$ Sine $\delta(K_3)=2$, $d\ge 2$, and $\delta(P_2)=1,$  
we have $\min\{\delta(K_3), d\} =2 > \delta(P_2)=1.$ Hence, all assumptions of Conjecture \ref{PConjK} are satisfied for $K_3$ and $P_2.$  Then the statement of the corollary follows immediately from Theorem \ref{mainthmS5}.
\end{proof}

\section{Mapping properties for the $(C_4+t)$-Operator}
We investigate the mapping properties of the operator associated with the graph $C_4$ + diagonal. 
Throughout the remaining sections, we assume that $t$ is a non-zero element in $\mathbb F_q^*.$
Let  $f_i, 1\le i\le 4,$ be non-negative  real-valued functions on $\mathbb F_q^d.$ 

The operator ${\diamondsuit}_t$ is associated with the graph $C_4$ + diagonal $t$ (Figure \ref{fig:sub2K}), and we define ${\diamondsuit}_t(f_1,f_2,f_3, f_4)$ as the quantity
\begin{equation}\label{Rhdef} \frac{1}{q^d|S_{t}||S_t^{d-2}|^2} \sum_{x^1, x^2, x^3, x^4\in \mathbb F_q^d} S_t(x^1-x^2) S_t(x^2-x^3) S_t(x^3-x^4)  S_t(x^4-x^1)  S_t(x^1-x^3) \prod_{i=1}^4 f_i(x^i).\end{equation}

The operator ${\diamondsuit}_t$ is referred to as the $(C_4+t)$-Operator on $\mathbb F_q^d.$
Here, notice that we take the quantity $q^d|S_{t}||S_t^{d-2}|^2$ as the normalizing factor $\mathcal{N}(G)$ in \eqref{generalform}.  

Applying a change of variables by letting $x=x^1,  u=x^1-x^2, v=x^1-x^3, w=x^1-x^4,$   we see that 
\begin{equation}\label{eq6.1K}{\diamondsuit}_t(f_1, f_2, f_3, f_4)=\frac{1}{q^d} \sum_{x\in \mathbb F_q^d} f_1(x) T(f_2,f_3,f_4)(x)=< f_1,~ T(f_2, f_3, f_4)>,\end{equation}
where the operator $T(f_2, f_3, f_4)$ is defined by 
 \begin{equation}\label{defBTri} T(f_2, f_3, f_4)(x):= \frac{1}{|S_t||S_t^{d-2}|^2}    \sum_{u,v, w\in S_t} S_t(v-u) S_t(w-v)   f_2(x-u) f_3(x-v)f_4(x-w).\end{equation}

\begin{definition} Let $1\le p_1, p_2, p_3, p_4\le \infty.$ 
We define ${\diamondsuit}_t(p_1, p_2, p_3, p_4)$ to be  the smallest constant such that   the following estimate holds  for all non-negative real-valued functions $f_i, 1\le i\le 4,$ on $\mathbb F_q^d:$
$ {\diamondsuit}_t(f_1,f_2, f_3, f_4) \le {\diamondsuit}_t(p_1, p_2, p_3, p_4) ||f_1||_{p_1} ||f_2||_{p_2} ||f_3||_{p_3} ||f_4||_{p_4}.$
\end{definition} 

 We are asked to find $1\le p_1, p_2, p_3, p_4\le \infty$ such that  
\begin{equation}\label{DeltaPP} {\diamondsuit}_t(f_1, f_2, f_3, f_4) \lesssim ||f_1||_{p_1} ||f_2||_{p_2} ||f_3||_{p_3} ||f_4||_{p_4}\end{equation}
holds for all non-negative real-valued functions $f_i,1\le i\le 4,$ on $\mathbb F_q^d.$  
In other words,  our main problem is to determine  all numbers $1\le p_1, p_2, p_3, p_4 \le \infty$ such that
$ {\diamondsuit}_t(p_1, p_2, p_3, p_4) \lesssim 1.$

\begin{lemma} [Necessary conditions for the boundedness of ${\diamondsuit}_t(p_1, p_2, p_3, p_4)$  ]\label{N2244} 
Let ${\diamondsuit}_t$ be the $(C_4+t)$-Operator on $\mathbb F_q^d.$
If  ${\diamondsuit}_t(p_1, p_2, p_3, p_4) \lesssim 1$, then we have
$$ \frac{1}{p_1} +\frac{1}{p_2}+\frac{d}{p_3}+\frac{1}{p_4} \le d, \quad\frac{d}{p_1} +\frac{1}{p_2}+\frac{1}{p_3}+\frac{1}{p_4}\le d,\quad\mbox{and}\quad \frac{1}{p_1} +\frac{d}{p_2}+\frac{1}{p_3}+\frac{d}{p_4}\le 2d-2. $$
Also, under this assumption when $d = 2,$ it can be shown by Polymake \cite{AGHJLP, GJ} that $(1/p_1, 1/p_2, 1/p_3, 1/p_4)$ is contained in the convex hull of the points $(0,0,1,0),$ $(0,1,0,0),$ $(0,0,0,1),$ $(1/2,0,1/2,1/2),$ $(2/3,2/3,0,0),$ $(1,0,0,0),$ $(2/3,0,2/3,0),$ $(1/2,1/2,1/2,0),$ $(2/3,0,0,2/3),$ $(0,2/3,2/3,0),$ $(0,0,0,0),$ $(0,0,2/3,2/3).$
\end{lemma}

\begin{proof}  We prove the first conclusion that
\begin{equation}\label{Fe1}\frac{1}{p_1} +\frac{1}{p_2}+\frac{d}{p_3}+\frac{1}{p_4} \le d.\end{equation}

We notice from \eqref{eq6.1K} and \eqref{defBTri} that ${\diamondsuit}_t(f_1,f_2,f_3,f_4)$ becomes
\begin{equation*} \frac{1}{q^d} \sum_{x\in \mathbb F_q^d} f_1(x) \left[ \frac{1}{|S_t||S_t^{d-2}|^2}    \sum_{v\in S_t} f_3(x-v) \left( \sum_{u, w\in S_t} S_t(v-u) S_t(w-v)   f_2(x-u) f_4(x-w)\right)\right].\end{equation*}
 
Taking $f_1=f_2=f_4=1_{S_t},$ and $f_3=\delta_0,$   we see that
$$ ||f_1||_{p_1} ||f_2||_{p_2} |||f_3||_{p_3} ||f_4||_{p_4}  \sim q^{-\frac{1}{p_1}}q^{-\frac{1}{p_2}}q^{-\frac{d}{p_3}}   q^{-\frac{1}{p_4}}$$
and 
$$ {\diamondsuit}_t(f_1,f_2,f_3, f_4) = \frac{1}{q^d} \sum_{x\in S_t}  \frac{1}{|S_t||S_t^{d-2}|^2} \left(\sum_{u\in S_{t}: ||x-u||=t} 1\right)\left(\sum_{w\in S_{t}: ||x-w||=t} 1\right) \sim q^{-d},$$ 
which yields the required necessary condition \eqref{Fe1} ,  where 
the last similarity above  follows  from Corollary \ref{KKo} in Appendix.

The second conclusion follows by symmetry from the first conclusion.

To prove the third conclusion that is
$\frac{1}{p_1} +\frac{d}{p_2}+\frac{1}{p_3}+\frac{d}{p_4}\le 2d-2,$ 
we test the inequality \eqref{DeltaPP} with $f_1=f_3=1_{S_t}$ and $ f_2=f_4=\delta_0.$ 

Then we have
$$ ||f_1||_{p_1}||f_2||_{p_2} ||f_3||_{p_3} ||f_4||_{p_4} = \left(\frac{|S_t|}{q^d}\right)^{1/p_1}  \left(\frac{1}{q^d}\right)^{1/p_2}  \left(\frac{|S_t|}{q^d}\right)^{1/p_3} \left(\frac{1}{q^d}\right)^{1/p_4} \sim  q^{- \frac{1}{p_1} - \frac{d}{p_2} - \frac{1}{p_3}-\frac{d}{p_4}}.$$

 On the other hand, taking  $f_1=f_3=1_{S_t}$ and $ f_2=f_4=\delta_0$ in the definition \eqref{Rhdef},  we see that
$${\diamondsuit}_t(f_1, f_2, f_3, f_4)  =\frac{1}{q^d|S_{t}||S_t^{d-2}|^2} \sum_{x^1\in S_t}  \left(\sum_{x^3\in S_t}S_t(x^1-x^3)\right)\lesssim \frac{1}{q^d|S_t^{d-2}|}\sim q^{-2d+2}.$$
Hence, by \eqref{DeltaPP} we must have the required third conclusion that 
$ \frac{1}{p_1} +\frac{d}{p_2}+\frac{1}{p_3}+\frac{d}{p_4}\le 2d-2.$
\end{proof}

\subsection{Boundedness results for ${\diamondsuit}_t$ on $\mathbb F_q^d$}

Given a rhombus with a fixed diagonal(the graph $C_4$ + diagonal),  we will show that   by removing the vertex $x^2$ or  the vertex $x^4,$ 
$$ {\diamondsuit}_t(f_1,f_2,f_3, f_4) \lesssim ||f_2||_\infty  \Delta(f_1, f_3, f_4)\quad \mbox{and} \quad  {\diamondsuit}_t(f_1,f_2,f_3, f_4) \lesssim ||f_4||_\infty  \Delta(f_1, f_2, f_3).$$
Hence,   upper bounds of ${\diamondsuit}_t(p_1, \infty, p_3, p_4)$ and ${\diamondsuit}_t(p_1, p_2, p_3, \infty)$ can be controlled by upper bounds of the  $\Delta(p_1, p_3, p_4)$ and $\Delta(p_1, p_2, p_3),$ respectively. More precisely, we have the following relation.
\begin{proposition}\label{TriDia} Suppose that $\Delta(p, s, r) \lesssim 1 $ for $1\le p, s, r\le \infty.$ Then we have
$$ {\diamondsuit}_t(p, \infty, s, r)\lesssim 1 \quad \mbox{and}\quad {\diamondsuit}_t(p, s, r, \infty)\lesssim 1.$$
\end{proposition}
\begin{proof}
Since $\Delta(p, s, r) \lesssim 1 $ for $1\le p, s, r\le \infty,$   we see that for all non-negative functions $ f, g, h$ on $\mathbb F_q^d,$  
$ \Delta(f, g, h) \lesssim  ||f||_p ||g||_s ||h||_r.$
Thus, to complete the proof,  it will be enough to establish the following estimates: for all non-negative functions $f_i, i=1,2,3, 4,$ 
\begin{equation} \label{rhendnjs1}
{\diamondsuit}_t(f_1,f_2,f_3, f_4) \lesssim ||f_2||_\infty  \Delta(f_1, f_3, f_4)
\end{equation}
and
\begin{equation}\label{rhendnjs2}
{\diamondsuit}_t(f_1,f_2,f_3, f_4) \lesssim ||f_4||_\infty  \Delta(f_1, f_2, f_3).
\end{equation}

Since the proofs of both \eqref{rhendnjs1} and \eqref{rhendnjs2} are the same, we only provide the proof of  the estimate \eqref{rhendnjs2}.  Notice by the definition of ${\diamondsuit}_t(f_1,f_2,f_3, f_4)$ in \eqref{Rhdef}  that  ${\diamondsuit}_t(f_1,f_2,f_3,f_4)$ can be written as the form
$$\frac{1}{q^d|S_{t}||S_t^{d-2}|} \sum_{\substack{x^1, x^2, x^3\in \mathbb F_q^d\\
: ||x^1-x^2||=||x^2-x^3||= ||x^1-x^3||=t}} \left( \prod_{i=1}^3 f_i(x^i)\right) \left[\frac{1}{|S_t^{d-2}|}\sum_{x^4\in \mathbb F_q^d}S_t(x^3-x^4)  S_t(x^4-x^1)  f_4(x^4)\right].  $$

For each $x^1, x^3 \in \mathbb F_q^d$ with $||x^1-x^3||=t,$   we  define  $M(x^1, x^3)$ as the value in the above bracket. 
Then, recalling the definition of $\Delta(f_1,f_2,f_3) $ given in \eqref{KDeltatD},   we see that
$${\diamondsuit}_t(f_1,f_2,f_3,f_4) \le  \left(\max_{\substack{x^1, x^3\in \mathbb F_q^d\\
:||x^1-x^3||=t}} M(x^1, x^3) \right) \Delta(f_1,f_2,f_3).$$
Hence,  the estimate \eqref{rhendnjs2} follows immediately by proving the following claim:
\begin{equation}\label{claim1}
M:=\max_{\substack{x^1, x^3\in \mathbb F_q^d\\
:||x^1-x^3||=t}}  \frac{1}{|S_t^{d-2}|}\sum_{x^4\in \mathbb F_q^d}S_t(x^3-x^4)  S_t(x^4-x^1)  \lesssim 1.
\end{equation}
To prove this claim,  we first apply a change of variables by letting $x=x^1,  y=x^1-x^3.$ Then it follows that
$$ M=\max_{x\in \mathbb F_q^d, y\in S_t}  \frac{1}{|S_t^{d-2}|}\sum_{x^4\in \mathbb F_q^d}S_t(x-y-x^4)  S_t(x^4-x)  .$$
Letting $z=x-x^4,$   we have
$$ M=\max_{x\in \mathbb F_q^d, y\in S_t}  \frac{1}{|S_t^{d-2}|} \sum_{z\in S_t: ||z-y||=t} 1 \sim \frac{1}{q^{d-2}}\max_{y\in S_t} \sum_{z\in S_t: ||z-y||=t} 1$$
By Corollary \ref{KKo} in Appendix,  we conclude  that $M\lesssim 1,$ as required.
\end{proof}

In arbitrary dimensions $d\ge 2,$ we have the following consequences.

\begin{theorem}  \label{mainAT}
 Suppose that  $1\le a, b\le \infty$ satisfy that  
\begin{equation}\label{eqstart}\frac{1}{a} +  \frac{d}{b}\le d \quad \mbox{and} \quad   \frac{d}{a} +  \frac{1}{b}\le d.\end{equation}
Namely,  let $(1/a, 1/b)$ be contained in the convex hull of points $(0,0),$ $(0,1),$  $(d/(d+1), d/(d+1)),$ $(1,0).$
Then we have
${\diamondsuit}_t(a, \infty, b, \infty)\lesssim 1,$ ${\diamondsuit}_t(a, b,\infty, \infty)\lesssim 1 ,$ $  {\diamondsuit}_t(a, \infty, \infty, b)\lesssim 1,$ $ {\diamondsuit}_t(\infty, \infty, a, b)\lesssim 1,$ $ {\diamondsuit}_t(\infty, a, b, \infty)\lesssim 1 .$
\end{theorem}
\begin{proof} From Theorem \ref{thmSi},   we know that  the assumption \eqref{eqstart} implies that
$ \Delta(a, b, \infty)\lesssim 1,$ $\Delta(a, \infty, b)\lesssim 1$ and  $\Delta(\infty, a, b)\lesssim 1.$ 
Hence, the statement of the theorem follows immediately  by combining  these and Proposition \ref{TriDia}.
\end{proof}

\subsection{ Sharp boundedness results up to endpoints for $\diamondsuit_t$ on $\mathbb F_q^2$}
In this subsection, we collect our boundedness results for the operator  $\diamondsuit_t$ in two dimensions.

\begin{theorem} \label{R1K}Let $\diamondsuit_t$ be the $(C_4+t)$-Operator on $\mathbb F_q^2.$ Let $1\le p_1, p_2, p_3\le \infty.$    
\begin{enumerate}
\item [(i)]  Suppose that $(p_1, p_2, p_3)\ne (2,2,2)$ satisfies the following equations:
$$\frac{2}{p_1} +  \frac{1}{p_2}+ \frac{1}{p_3}\le 2,  ~~\frac{1}{p_1} +  \frac{2}{p_2}+ \frac{1}{p_3}\le 2,~~\frac{1}{p_1} +  \frac{1}{p_2}+ \frac{2}{p_3}\le 2.$$
Then we have $ \diamondsuit_t(p_1, \infty, p_2, p_3)\lesssim 1$ and $ \diamondsuit_t(p_1, p_2, p_3, \infty)\lesssim 1.$

\item [(ii)] In addition,  we have $\diamondsuit_t(2, \infty, 2,2)\lessapprox 1$ and $\diamondsuit_t(2,2,2, \infty)\lessapprox 1,$
where $\lessapprox$ is used to denote that  the boundedness of $\diamondsuit_t$ holds for all indicator test functions.
\end{enumerate}
\end{theorem}
\begin{proof} Notice that Proposition \ref{TriDia} still holds after replacing $\lesssim$ by $\lessapprox.$  Hence,  the statement of the theorem is directly obtained by combining Proposition \ref{TriDia} and  Theorem  \ref{Best2D} (ii). 
\end{proof}

Theorem \ref{R1K} guarantees the sharp boundedness for the operator $\diamond_t$ up to endpoints. Indeed, we have the following result.
\begin{theorem}\label{KiteK} Let $\diamondsuit_t$ be the $(C_4+t)$-Operator on $\mathbb F_q^2.$ The necessary conditions for $\diamondsuit_t(p_1, p_2, p_3, p_4)\lesssim 1$ given in Lemma \ref{N2244}  are  sufficient except for the two points $(p_1, p_2, p_3, p_4)= (2,2,2, \infty),$ $(2, \infty, 2,2).$ 

In addition, we have
\begin{equation}\label{intk}  \diamondsuit_t(2, \infty, 2,2)\lessapprox 1 \quad\mbox{and} \quad \diamondsuit_t(2,2,2, \infty)\lessapprox 1 .\end{equation}
\end{theorem}
\begin{proof} The statement \eqref{intk} was already proven  in  Theorem \ref{R1K} (ii).  Hence, using the interpolation theorem and the second part of Lemma \ref{N2244}, the matter is reducing to proving
$\diamondsuit_t(p_1, p_2, p_3, p_4) \lesssim 1$ for the critical endpoints  $ (1/p_1, 1/p_2, 1/p_3, 1/p_4)$ including all the following points: 
$(0,0,0,0),$ $(1,0,0,0),$ $(0,1,0,0),$ $(0,0,1,0),$  $(0,0,0,1),$ $(2/3,2/3,0,0),$  $(2/3,0,2/3,0),$  $(2/3,0,0,2/3),$ $(0,2/3,2/3,0),$  $(0,0,2/3,2/3).$ 

In other words,  the proof will be complete by proving the following estimates: 
$\diamondsuit_t(\infty, \infty, \infty, \infty) \lesssim 1,$ 
~$\diamondsuit_t(1, \infty, \infty, \infty) \lesssim 1,$
~$\diamondsuit_t(\infty, 1, \infty, \infty) \lesssim 1,$
~$\diamondsuit_t(\infty, \infty, 1, \infty) \lesssim 1,$ 
~$\diamondsuit_t(\infty, \infty, \infty, 1) \lesssim 1,$
~$\diamondsuit_t(3/2, 3/2, \infty, \infty) \lesssim 1,$
~$\diamondsuit_t(3/2, \infty, 3/2, \infty) \lesssim 1,$ 
~$\diamondsuit_t(3/2, \infty, \infty, 3/2) \lesssim 1,$
~$\diamondsuit_t(\infty, 3/2, 3/2, \infty) \lesssim 1,$
~$\diamondsuit_t(\infty, \infty, 3/2, 3/2) \lesssim 1.$

However, by a direct computation,  these estimates follow immediately from Theorem \ref{R1K} (i).
\end{proof}

\section{Boundedness problem for the $C_4$-Operator}
Let $t\in \mathbb F_q^*.$  Given non-negative real-valued functions $f_i, 1\le i\le 4$, on $\mathbb F_q^d,$  we define $\Diamond(f_1,f_2, f_3, f_4)$ to be the following value:
\begin{equation}\label{DiaDef}  \frac{1}{q^d |S_t|^2 |S_t^{d-2}|} \sum_{x^1, x^2, x^3, x^4\in \mathbb F_q^d} S_t(x^1-x^2) S_t(x^2-x^3) S_t(x^3-x^4) S_t(x^4-x^1) \prod_{i=1}^4 f_i(x^i), \end{equation}
where the quantity $q^d |S_t|^2 |S_t^{d-2}|$ stands for the normalizing factor $\mathcal{N}(G)$ in \eqref{generalform} when $G=C_4.$
Since the operator $\Diamond$ is associated with the graph $C_4$, it is named as the $C_4$-Operator on $\mathbb F_q^d.$
 
\begin{definition} Let $1\le p_1, p_2, p_3, p_4\le \infty.$ 
We denote by $\Diamond(p_1, p_2, p_3, p_4)$   the smallest constant such that   the following estimate holds  for all non-negative real-valued functions $f_i, 1\le i\le 4,$ on $\mathbb F_q^d:$
$$ \Diamond(f_1,f_2, f_3, f_4) \le \Diamond(p_1, p_2, p_3, p_4) ||f_1||_{p_1} ||f_2||_{p_2} ||f_3||_{p_3} ||f_4||_{p_4}.$$\end{definition} 

Main problem is to find all exponents  $1\le p_1, p_2, p_3, p_4\le \infty$ such that   the inequality
\begin{equation}\label{DeltaPPB} \Diamond(f_1, f_2, f_3, f_4) \lesssim ||f_1||_{p_1} ||f_2||_{p_2} ||f_3||_{p_3} ||f_4||_{p_4}\end{equation}
holds for all non-negative real-valued functions $f_i, 1\le i\le 4,$ on $\mathbb F_q^d.$  
In other words, our main problem is to determine  all numbers $1\le p_1, p_2, p_3, p_4 \le \infty$ such that
$ \Diamond(p_1, p_2, p_3, p_4) \lesssim 1.$

\begin{lemma} [Necessary conditions for the boundedness of $\Diamond(p_1, p_2, p_3, p_4)$  ]\label{NeB}Suppose that \eqref{DeltaPPB} holds, namely $\Diamond(p_1, p_2, p_3, p_4) \lesssim 1.$ Then we have
$$ \frac{1}{p_1} +\frac{1}{p_2}+\frac{1}{p_3}+\frac{d}{p_4} \le d+1, \quad\frac{1}{p_1} +\frac{1}{p_2}+\frac{d}{p_3}+\frac{1}{p_4} \le d+1, \quad\frac{1}{p_1} +\frac{d}{p_2}+\frac{1}{p_3}+\frac{1}{p_4} \le d+1,$$
$$  \frac{d}{p_1} +\frac{1}{p_2}+\frac{1}{p_3}+\frac{1}{p_4} \le d+1, \quad \frac{d}{p_1} +\frac{1}{p_2}+\frac{d}{p_3}+\frac{1}{p_4} \le 2d-2, \quad\mbox{and}\quad \frac{1}{p_1} +\frac{d}{p_2}+\frac{1}{p_3}+\frac{d}{p_4} \le 2d-2. $$\\
In particular, when $d = 2,$ it can be shown by Polymake \cite{AGHJLP, GJ} that $(1/p_1, 1/p_2, 1/p_3, 1/p_4)$ is contained in the convex hull of the points $(0,0,1,0),$ $(0,0,0,1),$ $(0,1,0,0),$ $(2/3,0,0,2/3),$ $(2/3,2/3,0,0),$ $(1,0,0,0),$ $(0,0,0,0),$ $(0,2/3,2/3,0),$ $(0,0,2/3,2/3).$ 
\end{lemma}
\begin{remark}\label{remK} When $d=2, 3$,  the first four inequalities in the conclusion are not necessary.  We only need the last two. \end{remark}
\begin{proof}The 6 inequalities in the conclusion can be easily deduced by testing the inequality \eqref{DeltaPPB} with the following specific functions, respectively:
We leave the proofs to the readers.
\begin{tabbing}
\=1) \hspace{4cm} \=2) \hspace{4cm} \kill
1) $f_1=f_2=f_3=1_{S_t},$ and $f_4=\delta_0.$ \hspace{2mm}
2)  $f_1=f_2=f_4=1_{S_t},$ and $f_3=\delta_0.$\\
3) $f_1=f_3=f_4=1_{S_t},$ and $f_2=\delta_0.$\hspace{3mm}
4) $f_2=f_3=f_4=1_{S_t},$ and $f_1=\delta_0.$\\
5) $f_2=f_4=1_{S_t},$ and $f_1=f_3=\delta_0.$\hspace{3mm}
6) $f_1=f_3=1_{S_t},$ and $f_2=f_4=\delta_0.$
\end{tabbing}

\end{proof}

\subsection{ Boundedness results for the $C_4$-Operator $\Diamond$ on $\mathbb F_q^d$}

In this subsection, we provide some exponents $1\le p_i \le \infty, 1\le i\le 4,$ such that $\Diamond(p_1, p_2, p_3, p_4)\lesssim 1$ in the specific case when one of $p_i$ is $\infty.$  However, our result will correspond to all dimensions $d\ge 2.$ In general, it is very hard to deduce  non-trivial boundedness results for the $C_4$-Operator on $\mathbb F_q^d.$

We begin by observing that  an upper bound of $\Diamond(f_1,f_2,f_3,f_4)$ can be controlled by  estimating for both the $K_2$-Operator $L$ and the $P_2$-Operator $\Lambda.$ 
\begin{proposition}\label{Bo1}For all non-negative functions $f_i, i=1,2,3,4,$ on $\mathbb F_q^d, d\ge 2,$  we have
$$ \Diamond(f_1, f_2, f_3, f_4) \lesssim  \left\{\begin{array}{ll} 
&\left(\frac{1}{|S_t^{d-2}|} L(f_1f_3,~ f_2)  +  \Lambda(f_1, f_2, f_3) \right) ||f_4||_\infty,\\
&\left(\frac{1}{|S_t^{d-2}|} L(f_2f_4,~ f_1)  +  \Lambda(f_4, f_1, f_2)\right)  ||f_3||_\infty,\\
&\left(\frac{1}{|S_t^{d-2}|} L(f_1f_3,~ f_4)  +  \Lambda(f_3, f_4, f_1)\right)  ||f_2||_\infty,\\
&\left(\frac{1}{|S_t^{d-2}|} L(f_2f_4,~ f_3)  +  \Lambda(f_2, f_3, f_4)\right)  ||f_1||_\infty.\\
\end{array} \right.$$
\end{proposition}
\begin{proof} We only provide the proof of the first inequality, 
\begin{equation} \label{WantE}\Diamond(f_1, f_2, f_3, f_4) \lesssim \frac{1}{|S_t^{d-2}|} L(f_1f_3,~ f_2) ||f_4||_\infty +  \Lambda(f_1, f_2, f_3) ||f_4||_\infty,\end{equation}
since other inequalities can be easily proven in the same way by replacing the role of $f_4$ with $f_3, f_2, f_1,$ respectively.
By definition,  the value of   $\Diamond(f_1,f_2,f_3,f_4)$ is equal to
$$ \frac{1}{q^d |S_t|^2 |S_t^{d-2}|} \sum_{x^1, x^2, x^3\in \mathbb F_q^d} S_t(x^1-x^2) S_t(x^2-x^3) \left(\prod_{i=1}^3 f_i(x^i)\right) \left(\sum_{x^4\in \mathbb F_q^d} S_t(x^3-x^4) S_t(x^4-x^1) f_4(x^4)\right).$$

For fixed $x^1, x^3 \in \mathbb F_q^d,$  the sum in the above bracket can be estimated as follows:
$$ \sum_{x^4\in \mathbb F_q^d} S_t(x^3-x^4) S_t(x^4-x^1) f_4(x^4) \lesssim \left\{\begin{array}{ll}
|S_t| ||f_4||_\infty\quad &\mbox{if} \quad x^1=x^3,\\
 q^{d-2} ||f_4||_\infty \quad &\mbox{if}\quad  x^1\ne x^3. \end{array} \right.$$
 Notice that this estimates are easily obtained by invoking  Corollary \ref{KKo} in Appendix after using a change of variables.
 
Let $\Diamond(f_1,f_2,f_3,f_4) =\Diamond_1 + \Diamond_2,$
 where $\Diamond_1$ denotes the contribution to $\Diamond(f_1, f_2, f_3, f_4)$  when $x^1=x^3,$  and  $\Diamond_2$  does it when $x^1\ne x^3.$ Then it follows that
 $$ \Diamond_1 \lesssim  \frac{||f_4||_\infty}{q^d |S_t| |S_t^{d-2}|}  \sum_{x^1, x^2\in \mathbb F_q^d} S_t(x^1-x^2) (f_1f_3)(x^1) f_2(x^2) =\frac{1}{|S_t^{d-2}|} L(f_1f_3,~ f_2) ||f_4||_\infty,$$
 %and  
$$ \Diamond_2 \lesssim  \frac{||f_4||_\infty}{q^d|S_t|^2 }  \sum_{x^1, x^2, x^3\in \mathbb F_q^d:x^1\ne x^3} S_t(x^1-x^2) S_t(x^2-x^3) f_1(x^1) f_2(x^2) f_3(x^3)\lesssim  \Lambda(f_1, f_2, f_3) ||f_4||_\infty.$$
Hence, we obtain the required estimate \eqref{WantE}.
\end{proof}

In Proposition \ref{Bo1},  we obtained four different kinds of the upper bounds of the $\Diamond(f_1, f_2, f_3, f_4).$ Using each of them, we are able to deduce  exponents $p_1, p_2, p_3, p_4$ with $ \Diamond(p_1, p_2, p_3, p_4)\lesssim 1,$  where at least one of $p_j, j=1,2,3, 4,$  takes $\infty.$

The following result can be proven by applying the first upper bound of $\Diamond(f_1, f_2,f_3, f_4)$ in Proposition \ref{Bo1} together with 
Theorem  \ref{GenRL} and Theorem \ref{SharpNS}.

\begin{theorem}\label{BaBa} Let $\Diamond$ be defined on functions on $\mathbb F_q^d, d\ge 2.$  Let $1\le p_1, p_2, p_3\le \infty.$ 
Then the following statements are true.
\begin{enumerate}
\item [(i)] If  $\frac{1}{p_1}+\frac{d}{p_2}+ \frac{1}{p_3} \le d$ and $ \frac{d}{p_1}+\frac{1}{p_2}+ \frac{d}{p_3} \le d$, then $\Diamond(p_1, p_2, p_3, \infty)\lesssim 1. $
\item [(ii)] If $ \frac{d}{p_1}+\frac{1}{p_2}+ \frac{1}{p_4} \le d$ and $\frac{1}{p_1}+\frac{d}{p_2}+ \frac{d}{p_4} \le d$, then $\Diamond(p_1, p_2,  \infty, p_4)\lesssim 1. $
\item [(iii)] If  $\frac{1}{p_1}+\frac{1}{p_3}+ \frac{d}{p_4} \le d$ and $ \frac{d}{p_1}+\frac{d}{p_3}+ \frac{1}{p_4} \le d$, then $\Diamond(p_1, \infty, p_3, p_4)\lesssim 1. $
\item [(iv)]If  $\frac{1}{p_2}+\frac{d}{p_3}+ \frac{1}{p_4} \le d$ and $ \frac{d}{p_2}+\frac{1}{p_3}+ \frac{d}{p_4} \le d$, then $\Diamond(\infty, p_2, p_3, p_4)\lesssim 1. $
\end{enumerate}
\end{theorem}
\begin{proof} 
We will only provide the proof of the first part of the theorem since  the proofs of other parts are the same in the sense that  
the proof of the first part uses the first upper bound of Proposition \ref{Bo1} and the proofs of other parts can also use their corresponding upper bounds of Proposition \ref{Bo1}  to complete the proofs.

Let us start proving the first part of the theorem. 
To complete the proof,  we aim to show that for all non-negative functions $f_i, i=1,2,3,4,$ on $\mathbb F_q^d,$ 
$$ \Diamond(f_1, f_2, f_3, f_4) \lesssim ||f_1||_{p_1} ||f_2||_{p_2} ||f_3||_{p_3}||f_4||_\infty,$$
 whenever   the exponents $1\le p_1, p_2, p_3\le \infty$ satisfy the following conditions:
\begin{equation}\label{ConKK}
\frac{1}{p_1}+\frac{d}{p_2}+ \frac{1}{p_3} \le d \quad\mbox{and}\quad \frac{d}{p_1}+\frac{1}{p_2}+ \frac{d}{p_3} \le d. \end{equation}

By the first part of Proposition \ref{Bo1},  it follows that
$$ \Diamond(f_1, f_2, f_3, f_4) \lesssim \left(\frac{1}{|S_t^{d-2}|} L(f_1f_3,~ f_2)  +  \Lambda(f_1, f_2, f_3)\right) ||f_4||_\infty.$$
Therefore, under the assumptions \eqref{ConKK}, our problem is reducing to establishing the following two estimates:
\begin{equation} \label{Hun1}L(f_1f_3,~ f_2) \lesssim ||f_1||_{p_1} ||f_2||_{p_2} ||f_3||_{p_3},\end{equation}
 \begin{equation}\label{Hun2}  \Lambda(f_1, f_2, f_3)\lesssim ||f_1||_{p_1} ||f_2||_{p_2} ||f_3||_{p_3}.\end{equation}
 For $1\le p_1, p_3\le \infty, $  let  $1/r= 1/p_{1} + 1/p_{3}.$ Then the conditions \eqref{ConKK} are the same as
$$ \frac{1}{r} +\frac{d}{p_2} \le d \quad \mbox{and} \quad \frac{d}{r} + \frac{1}{p_2}\le d.$$
 So these conditions enable us to invoke Theorem \ref{GenRL} so that we obtain the estimate \eqref{Hun1} as follows:
 $$ L(f_1f_3,~ f_2) \lesssim ||f_1f_3||_r ||f_2||_{p_2} \le ||f_1|||_{p_1} ||f_2||_{p_2} ||f_3||_{p_3},$$
 where we used H\"older's inequality in the last inequality. 
 
 It remains to prove the estimate \eqref{Hun2} under the assumptions \eqref{ConKK}.
 To do this, we shall use Theorem \ref{SharpNS}, which gives sufficient conditions for $\Lambda(p_1, p_2, p_3)\lesssim 1.$
 We directly compare the conditions \eqref{ConKK} with  the assumptions of Theorem \ref{SharpNS}.
 Then it is not hard to observe the following statements.
 \begin{itemize}
 \item (Case 1)  In the case when $0\le \frac{1}{p_1}, \frac{1}{p_3} \le \frac{d}{d+1}$,    the conditions \eqref{ConKK}  imply  the hypothesis of the first part of Theorem \ref{SharpNS}. 
 \item (Case 2)  In the case when $0\le \frac{1}{p_1} \le  \frac{d}{d+1} \le \frac{1}{p_3} \le 1,$  the conditions \eqref{ConKK} imply the hypothesis of the second part of Theorem \ref{SharpNS}.
 To see this,  notice that  if $d/p_1+ 1/p_2+ d/p_3\le d$, then  $ 1/(dp_1)+ 1/p_2+ d/p_3\le d.$
  \item (Case 3)  In the case when $0\le \frac{1}{p_3} \le  \frac{d}{d+1} \le \frac{1}{p_1} \le 1,$  the conditions \eqref{ConKK}  imply the hypothesis of the third part of Theorem \ref{SharpNS}.
   \item (Case 4)  In the case when $\frac{d}{d+1}\le \frac{1}{p_1}, \frac{1}{p_3} \le 1,$  the conditions \eqref{ConKK}  imply the hypothesis of the fourth part of Theorem \ref{SharpNS}.
 \end{itemize}
 
 Hence,  we conclude from Theorem \ref{SharpNS} that   $\Lambda(p_1, p_2, p_3)\lesssim 1$ under the assumptions \eqref{ConKK}, as desired. 
 \end{proof}
%%%%%%%%%%%%%%%%%%%%%%%%%%%%%%%%%%%%%%%%%%%%%%%%%%%%%%%%%%%%%
%%%%%%%%%%%%%%%%%%%%%%%%%%%%%%%%%%%%%%%%%%%%%%%%%%%%%%%%%%%

\subsection{Sharp boundedness results for the $C_4$-Operator $\Diamond$ on $\mathbb F_q^2$}

Recall that Theorem \ref{BaBa} provides sufficient conditions for $\diamond(p_1, p_2, p_2, p_4)\lesssim 1$ in any dimensions $d\ge 2.$ In this section, we show that Theorem \ref{BaBa} is sharp in two dimensions. More precisely, using Theorem \ref{BaBa} we will prove the following optimal result.
\begin{theorem} \label{Thm7.6K}Let $\Diamond$ be the $C_4$-Operator on $\mathbb F_q^2.$  For $1\le p_i\le \infty, 1\le i\le 4,$ we have
$$ \Diamond(p_1, p_2, p_3, p_4)\lesssim 1 ~~\mbox{if and only if}\quad \frac{2}{p_1} +\frac{1}{p_2}+\frac{2}{p_3}+\frac{1}{p_4} \le 2, ~~\mbox{and}\quad \frac{1}{p_1} +\frac{2}{p_2}+\frac{1}{p_3}+\frac{2}{p_4} \le 2.$$
\end{theorem}

\begin{proof} The necessary conditions for $\Diamond(p_1, p_2,p_3, p_4)\lesssim 1$ follow immediately from Lemma \ref{NeB} for $d=2$ (see Remark \ref{remK}).

Conversely, suppose that $1\le p_1, p_2, p_3, p_4\le \infty$ satisfy the following two inequalities:
\begin{equation}\label{twoE}  
\frac{2}{p_1} +\frac{1}{p_2}+\frac{2}{p_3}+\frac{1}{p_4} \le 2, ~~\mbox{and}\quad \frac{1}{p_1} +\frac{2}{p_2}+\frac{1}{p_3}+\frac{2}{p_4} \le 2.
\end{equation}
Then, as mentioned in Lemma \ref{NeB}, it can be shown by Polymake \cite{AGHJLP, GJ} that $(1/p_1, 1/p_2, 1/p_3, 1/p_4)$ is contained in the convex hull of the points $(0,0,1,0),$ $(0,0,0,1),$ $(0,1,0,0),$ $(2/3,0,0,2/3),$ $(2/3,2/3,0,0),$ $(1,0,0,0),$ $(0,0,0,0),$ $(0,2/3,2/3,0),$ $(0,0,2/3,2/3).$ 

By interpolating the above 9 critical points,  to prove  $\Diamond(p_1, p_2, p_3, p_4)\lesssim 1 $ for all $p_i, 1\le i\le 4$ satisfying the inequalities in \eqref{twoE},  it will be enough to prove it for the 9 critical points $(1/p_1, 1/p_2, 1/p_3, 1/p_4).$
This can be easily proven by using Theorem \ref{BaBa}. For example, for the point $(1/p_1, 1/p_2, 1/p_3, 1/p_4)=(2/3,0,0,2/3),$ a direct computation shows that the assumptions in Theorem \ref{BaBa} (ii) are satisfied and thus  $\Diamond(p_1, p_2, p_3, p_4)=\Diamond(3/2, \infty, \infty, 3/2) \lesssim 1.$ 
For other critical points,  we can easily prove them in the same way so that we omit the detail proofs.
\end{proof}

Notice that the graph $C_4$ is a subgraph of the graph $C_4$ + diagonal, and they are associated with the operators $\Diamond$ and $\diamondsuit_t,$ respectively. Hence, the following theorem shows that the answer to Question \ref{PConj} is negative when $G$ is the $C_4$ + diagonal, and $G'$ is the $C_4.$  However, this does not mean that Conjecture \ref{PConjK} is not true since  the $C_4$ and  the $C_4$ + diagonal do not satisfy  the main hypothesis \eqref{mainConj} of Conjecture \ref{PConjK}. 
\begin{theorem} \label{CExa}
Let $\diamondsuit_t, \Diamond$ be the $(C_4+t)$-Operator and the $C_4$-Operator on $\mathbb F_q^2,$ respectively. Let $1\le p_1, p_2, p_3, p_4\le \infty.$ Then the following statements hold.
\begin{itemize}
\item [(i)] If $\Diamond(p_1, p_2, p_3, p_4)\lesssim 1$, then $\diamondsuit_t(p_1, p_2, p_3, p_4)\lesssim 1.$
\item [(ii)] Moreover, there exist exponents $1\le a, b,c, d\le \infty$ such that $\diamondsuit_t(a, b, c, d)\lesssim 1$ but $\Diamond(a, b, c, d)$ is not bounded. 
\end{itemize}
\end{theorem}
\begin{proof}
First, let us prove the statement (ii) in the conclusion. To prove this,  we choose 
$(a, b, c, d)=(3/2, \infty, 3/2, \infty).$
From Theorem \ref{R1K} (i), we can easily note that $\diamondsuit_t(3/2, \infty, 3/2, \infty) \lesssim 1.$
However, it is impossible that $\Diamond(3/2, \infty, 3/2, \infty) \lesssim 1,$ which can be shown from Theorem \ref{Thm7.6K}.

Next, let us prove the first conclusion of the theorem.
Suppose that $\Diamond(p_1, p_2, p_3, p_4)\lesssim 1$ for $1\le p_i\le \infty, 1\le i\le 4.$ Then, as mentioned in the second conclusion of Lemma \ref{NeB},  the point $(1/p_1, 1/p_2, 1/p_3, 1/p_4)$ lies on the convex body with the critical endpoints: $(0,0,1,0),$ $(0,0,0,1),$ $(0,1,0,0),$ $(2/3,0,0,2/3),$ $(2/3,2/3,0,0),$ $(1,0,0,0),$ $(0,0,0,0),$ $(0,2/3,2/3,0),$ $(0,0,2/3,2/3).$ 

Invoking the interpolation theorem,  to prove the conclusion that $\diamondsuit_t(p_1, p_2, p_3, p_4)\lesssim 1,$ it will be enough to establish the boundedness only for those 9 critical points $(1/p_1, 1/p_2, 1/p_3, 1/p_4).$ More precisely, it remains to establish the following estimates:

$ \diamondsuit_t(\infty, \infty, \infty, \infty)\lesssim 1,$ 
~$ \diamondsuit_t(1, \infty, \infty, \infty)\lesssim 1,$
~$\diamondsuit_t(\infty, 1, \infty, \infty)\lesssim 1,$
~$\diamondsuit_t(\infty, \infty, 1, \infty)\lesssim 1,$\\
~$\diamondsuit_t(\infty, \infty, \infty, 1)\lesssim 1,$
~$\diamondsuit_t(3/2, 3/2,\infty, \infty)\lesssim 1,$
~$\diamondsuit_t(3/2, \infty, \infty, 3/2)\lesssim 1,$
~$\diamondsuit_t(\infty, 3/2, 3/2, \infty)\lesssim 1,$
$\diamondsuit_t(\infty, \infty, 3/2, 3/2)\lesssim 1.$

However, these estimates follow by applying  Theorem \ref{R1K} (i).
\end{proof}

\section{ Boundedness problem for the $P_3$-Operator}
For $t\in \mathbb F_q^*$ and  non-negative real-valued functions $f_i, i=1,2,3, 4,$ on $\mathbb F_q^d,$  we define $\sqcap(f_1,f_2, f_3, f_4)$ as the following value:
\begin{equation} \label{defP3}\frac{1}{q^d|S_t|^3}    \sum_{x^1, x^2, x^3, x^4\in \mathbb F_q^d} S_t(x^1-x^2) S_t(x^2-x^3) S_t(x^3-x^4) \prod_{i=1}^4 f_i(x^i). \end{equation}
This operator $\sqcap$ will be named the $P_3$-Operator on $\mathbb F_q^d$ since
it is related to the graph $P_3$ with vertices in $\mathbb F_q^d, d\ge 2.$ Note that in the definition of $\sqcap(f_1,f_2, f_3, f_4)$, we take the normalizing fact $q^d|S_t|^3$, which is corresponding to $\mathcal{N}(G)$ in \eqref{generalform} when $G$ is the $P_3.$  
\begin{definition} Let $1\le p_1, p_2, p_3, p_4\le \infty.$ 
We define $\sqcap(p_1, p_2, p_3, p_4)$ as the smallest constant such that   the following estimate holds  for all non-negative real-valued functions $f_i, i=1,2,3, 4,$ on $\mathbb F_q^d:$
$$ \sqcap(f_1,f_2, f_3, f_4) \le \sqcap(p_1, p_2, p_3, p_4) ||f_1||_{p_1} ||f_2||_{p_2} ||f_3||_{p_3} ||f_4||_{p_4}.$$
\end{definition} 

 We want  to determine $1\le p_1, p_2, p_3, p_4\le \infty$ such that  
\begin{equation}\label{Loo} \sqcap(f_1, f_2, f_3, f_4) \lesssim ||f_1||_{p_1} ||f_2||_{p_2} ||f_3||_{p_3} ||f_4||_{p_4}\end{equation}
holds for all non-negative real-valued functions $f_i, i=1,2,3,4,$ on $\mathbb F_q^d.$  
In other words,  our main problem is to find  all numbers $1\le p_1, p_2, p_3, p_4 \le \infty$ such that
$ \sqcap(p_1, p_2, p_3, p_4) \lesssim 1.$

\begin{lemma} [Necessary conditions for $\sqcap(p_1, p_2, p_3, p_4)\lesssim 1$]\label{NCap}
Suppose that  $\sqcap(p_1, p_2, p_3, p_4) \lesssim 1.$ Then we have
$\frac{1}{p_1} +\frac{d}{p_2}+\frac{1}{p_3} \le d,$ 
~ $\frac{1}{p_2}+\frac{d}{p_3}+\frac{1}{p_4} \le d,$
~ $\frac{d}{p_1} +\frac{1}{p_2}+\frac{1}{p_3}+\frac{1}{p_4} \le d+2,$
~$\frac{1}{p_1} +\frac{1}{p_2}+\frac{1}{p_3}+\frac{d}{p_4} \le d+2,$ 
~$\frac{1}{p_1} +\frac{d}{p_2}+\frac{1}{p_3}+\frac{d}{p_4} \le 2d-1,$
~ $\frac{d}{p_1} +\frac{1}{p_2}+\frac{d}{p_3}+\frac{1}{p_4} \le 2d-1,$ 
and $\frac{d}{p_1} +\frac{1}{p_2}+\frac{1}{p_3}+\frac{d}{p_4} \le 2d $

In particular,  when $d = 2$, by using Polymake \cite{AGHJLP, GJ}, it can be shown  that $(1/p_1, 1/p_2, 1/p_3, 1/p_4)$ is contained in the convex hull of the points: $(0,1,0,1/2),$ $(0,1,0,0),$ $(1/2,0,1/2,1),$ $(1/2,0,1,0),$ $ (1,1/2,0,0),$ $(1,0,0,0),$ $(1,1/3,1/3,0),$ $(1,0,1/2,0),$ $(1/2,1/3,5/6,0),$ $(1,0,0,1),$ $(1/2,1/2,1/2,1/2),$ $(1,1/2,0,1/2),$ $(0,5/6,1/3,1/2),$ $(0,0,0,1),$ $(0,1/2,0,1),$ $(0,2/3,2/3,0),$ $(0,1/3,1/3,1),$ $(0,0,1/2,1),$ $(0,0,1,0),$ $(0,0,0,0).$ 
\end{lemma}
\begin{remark}\label{remKK} When $d=2$,  the third, fourth, and seventh inequalities above are not necessary. When $d=3$,  the third and fourth inequalities above are not necessary. \end{remark}
\begin{proof} 
As in the proofs of Propositions \ref{NesLam}, \ref{NCD123}, and \ref{NecLam}, the conclusions of the statement follow by testing the inequality \eqref{Loo} with the following specific functions, respectively:
\begin{tabbing}
\=1) \hspace{4cm} \=2) \hspace{4cm} \kill
1) $f_1=f_3=1_{S_t},  f_2=\delta_0, $ and $f_4=1_{\mathbb F_q^d}.$\hspace{3mm}
2)  $f_1=1_{\mathbb F_q^d}, f_2=f_4=1_{S_t},$ and $f_3=\delta_0.$
\\
3) $f_2=f_3=f_4=1_{S_t},$ and $f_1=\delta_0.$\hspace{10mm}
4) $f_1=f_2=f_3=1_{S_t},$ and $f_4=\delta_0.$\\
5) $f_1=f_3=1_{S_t},$ and $f_2=f_4=\delta_0.$\hspace{10mm}
6)  $f_2=f_4=1_{S_t},$ and $f_1=f_3=\delta_0.$ \\
7)  $f_2=f_3=1_{S_t},$ and $f_1=f_4=\delta_0.$
\end{tabbing}
\end{proof}

\subsection{Boundedness results for the $P_3$-Operator $\sqcap$ on $\mathbb F_q^d$}

We begin by observing that  an upper bound of $\sqcap(f_1, f_2, f_3, f_4)$ can be controlled by the value  $\Lambda(f_1, f_2, f_3).$
 
\begin{proposition}\label{RelLH}  Let $1\le a, b,c \le \infty.$ 
If $\Lambda(a,b,c)\lesssim 1$, then  $\sqcap(a,b, c, \infty)\lesssim 1$ and  $\sqcap(\infty, a,b, c)\lesssim 1.$
\end{proposition}

\begin{proof} For all non-negative functions $f_i, i=1,2,3,4,$ on $\mathbb F_q^d,$  our task is to prove the following inequalities:
\begin{equation}\label{fir2} \sqcap(f_1, f_2,f _3, f_4) \lesssim \left\{\begin{array}{ll} & \Lambda(f_1,f_2, f_3) ~||f_4||_\infty, \\
 & ||f_1||_\infty ~\Lambda(f_2, f_3, f_4).\end{array} \right.\end{equation}
 We will only prove the first  inequality, that is
 \begin{equation} \label{Kaim} \sqcap(f_1, f_2,f _3, f_4) \lesssim \Lambda(f_1,f_2, f_3) ~||f_4||.\end{equation}
 By symmetry, the second inequality can be easily proven in the same way.
 By definition in \eqref{defP3},  we can write  $ \sqcap(f_1, f_2,f _3, f_4) $ as 
 $$\frac{1}{q^d |S_t|^2} \sum_{x^1, x^2, x^3\in \mathbb F_q^d} S_t(x^1-x^2) S_t(x^2-x^3) \left(\prod_{i=1}^3 f_i(x^i)\right) \left(\frac{1}{|S_t|} \sum_{x^4\in \mathbb F_q^d}  f_4(x^4) S_t(x^3-x^4)\right).$$
 Since the value in the above bracket  is  $ Af_4(x^3)$, which is clearly dominated by $||Af_4||_\infty$,   the required estimate \eqref{Kaim} follows immediately from the definition of $\Lambda(f_1,f_2, f_3)$ in \eqref{DefHinges}.
\end{proof}

The following theorem can be deduced from Proposition \ref{RelLH} and Theorem \ref{SharpNS}.
\begin{theorem}\label{thm8.5} Consider the $P_3$-Operator $\sqcap$ on $\mathbb F_q^d.$ Suppose that  the exponents $1\le a, b, c\le \infty$ satisfy one of the following conditions:
\begin{enumerate}
\item [(i)] $0\le \frac{1}{a}, \frac{1}{c} \le \frac{d}{d+1}$ and $ \frac{1}{a}+\frac{d}{b}+ \frac{1}{c} \le d$
\item [(ii)] $0\le \frac{1}{a} \le  \frac{d}{d+1} \le \frac{1}{c} \le 1,$ and $ \frac{1}{da}+\frac{1}{b}+ \frac{d}{c} \le d$
\item [(iii)] $0\le \frac{1}{c} \le  \frac{d}{d+1} \le \frac{1}{a} \le 1,$ and $ \frac{d}{a}+\frac{1}{b}+ \frac{1}{dc} \le d$ 
\item [(iv)]$\frac{d}{d+1}\le \frac{1}{a}, \frac{1}{c} \le 1$ and $ \frac{d}{a}+\frac{1}{b}+ \frac{d}{c} \le 2d-1.$
\end{enumerate}
Then we have
$\sqcap(a,b, c, \infty)\lesssim 1$ and $\sqcap(\infty, a,b, c)\lesssim 1.$
\end{theorem}
\begin{proof}
Using Theorem \ref{SharpNS} with $p_1=a, ~p_2=b,~ p_3=c$,  it is clear that  $\Lambda(a,b, c)\lesssim 1$ for all exponents $a, b, c$ in our assumption. Hence, the statement follows immediately from  Proposition \ref{RelLH}.
\end{proof}

Now we prove that  the value $\sqcap(f_1,f_2,f_3, f_4)$ can be expressed in terms of the averaging operator over spheres.
For functions $f, g, h$ on $\mathbb F_q^d$,   let us denote
$$ <f, g, h>:=||fgh||_1=\frac{1}{q^d} \sum_{x\in \mathbb F_q^d} f(x) g(x) h(x).$$
\begin{proposition} \label{FormPI}
Let $f_i, i=1,2,3, 4,$ be  non-negative real-valued functions on $\mathbb F_q^d.$ Then we have
$$ \sqcap(f_1, f_2, f_3, f_4)=<Af_1,~ f_2, ~ A(f_3\cdot Af_4) >  = <A( f_2\cdot Af_1), ~f_3,~ Af_4 >.$$
\end{proposition} 
\begin{proof}
By symmetry, to complete the proof, it suffices to prove the first equality, that is
$$ \sqcap(f_1, f_2, f_3, f_4)=<Af_1,~ f_2, ~ A(f_3\cdot Af_4) >.$$

Combining the definition in \eqref{defP3} and the definition of the spherical averaging operator $A$,  it follows that
\begin{align*}\sqcap(f_1, f_2, f_3, f_4) &= \frac{1}{q^d} \sum_{x^2\in \mathbb F_q^d} f_2(x^2) Af_1(x^2)  \left[ \frac{1}{|S_t|} \sum_{x^3\in \mathbb F_q^d} f_3(x^3) S_t(x^2-x^3)  Af_4(x^3) \right]\\
&=\frac{1}{q^d} \sum_{x^2\in \mathbb F_q^d} f_2(x^2) Af_1(x^2) A( f_3\cdot Af_4)(x^2).\end{align*}
This gives the required estimate.
\end{proof}

Combining Proposition \ref{FormPI} and the averaging estimate over spheres,  we are able to deduce sufficient conditions for the boundedness of the $P_3$-Operator $\sqcap$ on $\mathbb F_q^d.$ 
\begin{theorem} \label{thm8.7} Let  $1\le p_1,p_2, p_3, p_4\le \infty$ be exponents satisfying one of the following conditions:
\begin{enumerate}
\item [(i)] $0\le \frac{1}{p_1}, \frac{1}{p_4}, ~\frac{1}{p_3}+\frac{1}{dp_4}\le \frac{d}{d+1},$ and $  \frac{1}{dp_1}+\frac{1}{p_2}+\frac{1}{dp_3}+\frac{1}{d^2p_4} \le 1.$

\item [(ii)]$0\le \frac{1}{p_1}, \frac{1}{p_4} \le \frac{d}{d+1} \le \frac{1}{p_3}+\frac{1}{dp_4}\le 1,$ and $  \frac{1}{dp_1}+\frac{1}{p_2}+\frac{d}{p_3}+\frac{1}{p_4} \le d.$

\item [(iii)]$0\le \frac{1}{p_1}, ~\frac{1}{p_3}+\frac{d}{p_4} -d+1\le \frac{d}{d+1} \le \frac{1}{p_4}\le 1,$ and $  \frac{1}{p_1}+\frac{d}{p_2}+\frac{1}{p_3}+\frac{d}{p_4} \le 2d-1.$
 
\item [(iv)]$0\le \frac{1}{p_1} \le \frac{d}{d+1} \le \frac{1}{p_4},~\frac{1}{p_3}+\frac{d}{p_4}-d+1\le 1,$ and $  \frac{1}{dp_1}+\frac{1}{p_2}+\frac{d}{p_3}+\frac{d^2}{p_4} \le d^2.$

\item [(v)] $0\le \frac{1}{p_4},~\frac{1}{p_3}+\frac{1}{dp_4}\le \frac{d}{d+1} \le \frac{1}{p_1}\le 1,$ and $  \frac{d}{p_1}+\frac{1}{p_2}+\frac{1}{dp_3}+\frac{1}{d^2p_4} \le d.$

\item [(vi)] $0\le \frac{1}{p_4}\le \frac{d}{d+1} \le \frac{1}{p_1},~\frac{1}{p_3}+\frac{1}{dp_4}\le 1,$  and $ \frac{d}{p_1}+\frac{1}{p_2}+\frac{d}{p_3}+\frac{1}{p_4} \le 2d-1.$

\item [(vii)] $0\le \frac{1}{p_3}+\frac{d}{p_4}-d+1\le \frac{d}{d+1}\le \frac{1}{p_1}, \frac{1}{p_4} \le 1,$  and 
 $  \frac{d^2}{p_1}+\frac{d}{p_2}+\frac{1}{p_3}+\frac{d}{p_4} \le d^2+d-1.$

\item [(viii)]  $\frac{d}{d+1}\le \frac{1}{p_1}, \frac{1}{p_4}, \frac{1}{p_3}+\frac{d}{p_4}-d+1\le 1,$  and 
 $  \frac{d}{p_1}+\frac{1}{p_2}+\frac{d}{p_3}+\frac{d^2}{p_4} \le d^2+d-1.$
\end{enumerate}
Then we have
$\sqcap(p_1,p_2, p_3, p_4)\lesssim 1$ and  $\sqcap(p_4,p_3, p_2, p_1)\lesssim 1.$
\end{theorem}
\begin{proof}
By symmetry, it will be enough to prove the first part of conclusions, that is
$ \sqcap(p_1, p_2, p_3, p_4) \lesssim 1.$
To complete the proof, we will first find the general conditions that guarantee this conclusion.  Next we will demonstrate that  each of the hypotheses in the theorem satisfies the general conditions. 

To derive the first general condition,  we assume  that $1\le r_1, p_2, r\le \infty$  satisfy that 
\begin{equation}\label{GenAss1} \frac{1}{r_1}+ \frac{1}{p_2}+ \frac{1}{r}\le1.\end{equation}
Then by Proposition \ref{FormPI} and  H\"older's inequality, 
$$\sqcap(f_1,f_2,f_3, f_4)  \le ||Af_1||_{r_1} ||f_2||_{p_2} ||A(f_3 \cdot Af_4)||_r, $$
where we also used the nesting property of norms associated with the normalizing counting measure.
Assume that  $1\le p_1, s \le \infty$ satisfy the following averaging estimates over spheres:
\begin{equation}\label{GenAss2} A(p_1 \to r_1) \lesssim 1 \quad \mbox{and}\quad  A(s \to r)\lesssim 1.\end{equation}
It follows that
$\sqcap(f_1, f_2, f_3, f_4) \lesssim ||f_1||_{p_1} ||f_2||_{p_2}  ||f_3 \cdot Af_4||_s.$
Now we assume that  $1\le p_3, t\le \infty$ satisfy that
\begin{equation}\label{GenAss3} \frac{1}{s}=\frac{1}{p_3}+\frac{1}{t}.\end{equation}
Then, by the  H\"older's inequality,  we see that
$$\sqcap(f_1, f_2, f_3, f_4) \lesssim ||f_1||_{p_1} ||f_2||_{p_2} ||f_3||_{p_3}  ||Af_4||_t.$$
Finally, if we assume that $1\le p_4\le \infty$ satisfies the following averaging estimate:
\begin{equation}\label{GenAss4} A(p_4 \to t)\lesssim 1,\end{equation}
then  we obtain that
$ \sqcap(f_1, f_2, f_3, f_4) \lesssim ||f_1||_{p_1} ||f_2||_{p_2} ||f_3||_{p_3} ||f_4||_{p_4}.$

In summary, we see that $\sqcap(p_1, p_2, p_3, p_4) \lesssim 1$ provided that  the numbers $1\le p_i\le \infty, i=1,2,3, 4,$ satisfy  all the conditions \eqref{GenAss1}, \eqref{GenAss2}, \eqref{GenAss3}, \eqref{GenAss4}. Thus, to finish the proof, we will show that  each of the 8 hypotheses in the theorem satisfies  all these conditions.

Given $1\le p_1, p_4\le \infty,$  by Lemma \ref{BALem}  we can chose  $1\le r_1, t\le \infty$ such that  the first averaging estimate in \eqref{GenAss2} and  the averaging estimate  \eqref{GenAss4} hold respectively.
More precisely, we can select $0\le 1/r_1,~1/ t\le 1$ as follows:
\begin{itemize}
\item   If  $0\le \frac{1}{p_1} \le \frac{d}{d+1}, $  then we take  $ 1/r_1= 1/(dp_1). $
\item   If  $0\le \frac{1}{p_4} \le \frac{d}{d+1},$  then we take  $ 1/t=1/(dp_4). $
\item   If  $\frac{d}{d+1}\le\frac{1}{p_1} \le 1,$  then  we choose  $1/r_1=  {d}/{p_1}-d+1.$
\item   If  $\frac{d}{d+1}\le\frac{1}{p_4} \le 1, $  then  we choose  $ 1/t={d}/{p_4}-d+1.$
\end{itemize}
In the next step,  we determine $1\le r\le \infty$ by using the condition \eqref{GenAss3} and  the second averaging estimate in \eqref{GenAss2}.
Since two kinds of $t$ values can be chosen as above,   the condition \eqref{GenAss3} becomes
$$\frac{1}{s}= \frac{1}{p_3}+ \frac{1}{dp_4} \quad \mbox{or}\quad  \frac{1}{s}=\frac{1}{p_3}+\frac{d}{p_4}-d+1.$$
Combining these  $s$ values with the second averaging estimate in \eqref{GenAss2},   the application of Lemma \ref{BALem} enables us to choose $1/r$ values as follows:
\begin{itemize}
\item  If $0\le \frac{1}{s}=\frac{1}{p_3}+ \frac{1}{dp_4} \le \frac{d}{d+1}$, then  we take $ \frac{1}{r}=\frac{1}{dp_3}+ \frac{1}{d^2p_4}.$
\item If $\frac{d}{d+1} \le  \frac{1}{s}=\frac{1}{p_3}+ \frac{1}{dp_4} \le 1,$ then we take $ \frac{1}{r}=\frac{d}{p_3}+ \frac{1}{p_4}-d+1.$
\item  If $0\le \frac{1}{s}=\frac{1}{p_3}+ \frac{d}{p_4}-d+1 \le \frac{d}{d+1}$, then we take  $ \frac{1}{r}=\frac{1}{dp_3}+ \frac{1}{p_4}-1+\frac{1}{d}.$
\item If $\frac{d}{d+1} \le \frac{1}{s}=\frac{1}{p_3}+ \frac{d}{p_4}-d+1  \le 1,$ then  we take $ \frac{1}{r}=\frac{d}{p_3}+ \frac{d^2}{p_4}-d^2+1.$
\end{itemize}
Finally,  use the condition \eqref{GenAss1} together with previously selected two values for $r_1$ and  four values for $r$.
Then we obtain the required remaining conditions.
\end{proof} 

\begin{remark} Notice that Theorem \ref{thm8.5} is a special case of Theorem \ref{thm8.7}. However, the proof of Theorem \ref{thm8.5} is much simpler than that of Theorem \ref{thm8.5}. 
\end{remark}

We do not know if the consequences from Theorems \ref{thm8.5} and \ref{thm8.7} imply the sharp boundedness results for the $P_3$-Operator $\sqcap$ on $\mathbb F_q^d.$ However, they play an important role in proving the theorem below, which states that the exponents for $\sqcap(p_1, p_2, p_3, p_4)\lesssim 1$  are less restricted than those for ${\diamondsuit}_t(p_1, p_2, p_3, p_4)\lesssim 1.$  The precise statement is as follows.

\begin{theorem}\label{ThmDtCap} Let ${\diamondsuit}_t$ and $\sqcap$ be the operators acting on the functions on $\mathbb F_q^2.$
If  ${\diamondsuit}_t(p_1, p_2, p_3, p_4)\lesssim 1$ for $1\le p_1, p_2, p_3, p_4\le \infty,$ 
then  $\sqcap(p_1, p_2, p_3, p_4)\lesssim 1.$
\end{theorem}
\begin{proof}
Assume that ${\diamondsuit}_t(p_1, p_2, p_3, p_4)\lesssim 1$ for $1\le p_1, p_2, p_3, p_4\le \infty.$ 
Then, by Lemma \ref{N2244}, the point  $(1/p_1, 1/p_2, 1/p_3, 1/p_4)$ is contained in the convex hull of  the following points: $(0,0,1,0),$ $(0,1,0,0),$ $(0,0,0,1),$ $(1/2,0,1/2,1/2),$ $(2/3,2/3,0,0),$ $(1,0,0,0),$ $(2/3,0,2/3,0),$ $(1/2,1/2,1/2,0),$ $(2/3,0,0,2/3),$ $(0,2/3,2/3,0),$ $(0,0,0,0),$ $(0,0,2/3,2/3).$

To complete the proof, by the interpolation theorem,  it suffices to show that for each of the above critical points  $(1/p_1, 1/p_2, 1/p_3, 1/p_4),$   we have
$$ \sqcap(p_1, p_2, p_3, p_4)\lesssim 1.$$
To prove this, we will use  Theorem \ref{thm8.7} and Theorem \ref{thm8.5}.
By Theorem \ref{thm8.7} with the hypothesis (i),  one can notice that   $\sqcap(3/2, \infty, \infty,  3/2)\lesssim 1,$ which is corresponding to the point   $(1/p_1, 1/p_2, 1/p_3, 1/p_4)=(2/3, 0,0,2/3).$ Similarly, Theorem \ref{thm8.7} with the hypothesis (ii) can be used for the point $(1/2,0,1/2,1/2)$, namely, $\sqcap(2, \infty, 2, 2)\lesssim 1.$

For any other points,  we can invoke Theorem \ref{thm8.5}. 
More precisely, we can  apply Theorem \ref{thm8.5} with the hypothesis (i) for  the points $(0,1,0,0),$ $(2/3,2/3,0,0),$ $(2/3, 0, 2/3, 0),$ $(1/2, 1/2, 1/2,0),$ $(0,2/3, 2/3, 0),$ $(0,0,0,0),$ $(0, 0, 2/3, 2/3).$  The  points $(0,0,1,0),$ $(0,0,0, 1)$ can be obtained by  Theorem \ref{thm8.5} with the  hypothesis (ii).  Finally, for the point $(1, 0, 0, 0)$, we can prove that  $\sqcap(1, \infty, \infty, \infty)\lesssim 1$ by using  Theorem \ref{thm8.5} with the hypothesis (iii). This completes the proof.
\end{proof}

\begin{remark} \label{remark8.10} The reverse statement of Theorem \ref{ThmDtCap} is not true in general. As a counterexample, we can take $p_1=3/2,~ p_2=3,~ p_3=3/2, p_4=\infty.$  Indeed,  the assumption (i) of Theorem \ref{thm8.5} with $d=2$ implies that
$\sqcap(3/2, 3, 3/2, \infty) \lesssim 1.$ However,
${\diamondsuit_t}(3/2, 3, 3/2, \infty)$ cannot be bounded, which follows from Lemma \ref{N2244}. 
\end{remark}

The following corollary is a consequence of Theorem \ref{ThmDtCap}.
\begin{corollary} \label{CorThmDtCap} Conjecture \ref{PConjK} is valid for the graph $C_4$ + diagonal and its subgraph $P_3$ in $\mathbb F_q^2.$
\end{corollary}
\begin{proof} It is obvious that the $P_3$ is a subgraph of $C_4$ + diagonal in $\mathbb F_q^2.$ For $d=2,$ it is plain to notice that $\min\{\delta(C_4 + diagonal), d\} =2 > \delta(P_3)=1.$
Thus, the graph $C_4$ + diagonal and its subgraph $P_3$ satisfy all assumptions of Conjecture \ref{PConjK}. Then the statement of the corollary follows immediately from Theorem \ref{ThmDtCap} since the operators ${\diamondsuit}_t$ and $\sqcap$ are related to the $C_4$ + diagonal and its subgraph $P_3$, respectively.
\end{proof}

The following theorem provides a concrete example for a positive answer to Question \ref{PConj} since the operators  $\Diamond$ and $\sqcap$ are related to the graph $C_4$ and its subgraph $P_3,$ respectively.  Furthermore, the graphs also satisfy Conjecture \ref{PConjK} (see Corollary \ref{CorThmBoxK} below). 
\begin{theorem}\label{ThmBoxK} Let $\Diamond$ and $\sqcap$ be the operators acting on the functions on $\mathbb F_q^2.$
If  $\Diamond(p_1, p_2, p_3, p_4)\lesssim 1,$ $1\le p_1, p_2, p_3, p_4\le \infty,$ 
then  $\sqcap(p_1, p_2, p_3, p_4)\lesssim 1.$
\end{theorem}
\begin{proof}
By Theorem \ref{CExa} (i), if $\Diamond(p_1, p_2, p_3, p_4)\lesssim 1$, then $\diamondsuit_t(p_1, p_2, p_3, p_4)\lesssim 1.$  By Theorem \ref{ThmDtCap},  if  
${\diamondsuit}_t(p_1, p_2, p_3, p_4)\lesssim 1,$ 
then  $\sqcap(p_1, p_2, p_3, p_4)\lesssim 1.$ Hence, the statement follows.
\end{proof}

\begin{remark} \label{remark8.11K} The reverse statement of Theorem \ref{ThmBoxK} cannot hold. As in Remark \ref{remark8.10},  if we can take $p_1=3/2, p_2=3, p_3=3/2, p_4=\infty,$ then $\sqcap(3/2, 3, 3/2, \infty) \lesssim 1.$ However, $\Diamond(3/2, 3, 3/2, \infty)$ cannot be bounded, which follows from Theorem \ref{Thm7.6K}. \end{remark}

\begin{corollary}\label{CorThmBoxK}
Conjecture \ref{PConjK} holds true for the graph $C_4$ and its subgraph $P_3$ on $\mathbb F_q^2.$
\end{corollary}
\begin{proof}
The main hypothesis \eqref{mainConj} of Conjecture \ref{PConjK} is satisfied for the the graph $C_4$ and its subgraph $P_3$ on $\mathbb F_q^2:$
$$ \min\{\delta(C_4), 2\}=2 >1= \delta(P_3).$$
Since the operators $\Diamond$ and $\sqcap$ are associated to the graph $C_4$ and its subgraph $P_3,$ respectively,  the statement of the corollary follows from Theorem \ref{ThmBoxK}. 
\end{proof}
\section{Operators associated to the graph $K_3$ + tail (a kite)}
Given $t\in \mathbb F_q^*$ and functions $f_i, i=1,2,3, 4,$ on $\mathbb F_q^d,$   we define $ \unlhd(f_1, f_2,f_3, f_4)$ as the following value:
$$\frac{1}{q^d |S_t|^2 |S_t^{d-2}|} \sum_{x^1, x^2, x^3, x^4\in \mathbb F_q^d} S_t(x^1-x^2) S_t(x^2-x^3) S_t(x^3-x^4) S_t(x^3-x^1) \prod_{i=1}^4 f_i(x^i). $$
Note that this operator $ \unlhd$ is related to the graph $K_3$ + tail (Figure \ref{fig:sub2KK}), and so the normalizing factor $\mathcal{N}(G)$ in \eqref{generalform} can be taken as the quantity $q^d |S_t|^2 |S_t^{d-2}|.$ 
\begin{definition} For $1\le p_1, p_2, p_3, p_4\le \infty,$ 
we define $\unlhd(p_1, p_2, p_3, p_4)$ as  the smallest constant such that   the following estimate holds  for all non-negative real-valued functions $f_i, i=1,2,3, 4,$ on $\mathbb F_q^d:$
$$ \unlhd(f_1,f_2, f_3, f_4) \le \unlhd(p_1, p_2, p_3, p_4) ||f_1||_{p_1} ||f_2||_{p_2} ||f_3||_{p_3} ||f_4||_{p_4}.$$
\end{definition} 

 Here, our main problem is to determine all exponents $1\le p_1, p_2, p_3, p_4\le \infty$ such that  
\begin{equation}\label{DeltaPPU} \unlhd(f_1, f_2, f_3, f_4) \lesssim ||f_1||_{p_1} ||f_2||_{p_2} ||f_3||_{p_3} ||f_4||_{p_4}\end{equation}
holds for all non-negative real-valued functions $f_i, i=1,2,3,4,$ on $\mathbb F_q^d.$  
In other words,  we are asked to determine  all numbers $1\le p_1, p_2, p_3, p_4 \le \infty$ such that
$ \unlhd(p_1, p_2, p_3, p_4) \lesssim 1.$

Recall that when $d=2$, we assume that $3\in \mathbb F_q$ is a square number.
\begin{lemma} [Necessary conditions for the boundedness of $\unlhd(p_1, p_2, p_3, p_4)$  ]\label{Neun}Suppose that \eqref{DeltaPPU} holds, namely $\unlhd(p_1, p_2, p_3, p_4) \lesssim 1.$ Then we have
$$ \frac{1}{p_1} +\frac{1}{p_2}+\frac{d}{p_3}+\frac{1}{p_4} \le d, \quad\frac{1}{p_1} +\frac{d}{p_2}+\frac{1}{p_3}\le d, \quad\frac{d}{p_1} +\frac{1}{p_2}+\frac{1}{p_3}\le d,$$
$$  \frac{1}{p_1} +\frac{d}{p_2}+\frac{1}{p_3}+\frac{d}{p_4} \le 2d-1, \quad\mbox{and}\quad \frac{d}{p_1} +\frac{1}{p_2}+\frac{1}{p_3}+\frac{d}{p_4}\le 2d-1. $$

In particular, if $d=2$, then it can be shown by Polymake \cite{AGHJLP, GJ} that $(1/p_1, 1/p_2, 1/p_3, 1/p_4)$ is contained in the convex hull of the  points:
$(0,0,1,0),$ $(0,1,0,0),$  $(0,1,0,1/2),$ $(2/3,0,2/3,0),$ $(1/2,1/2,1/2,0),$ $(5/6,0,1/3,1/2),$ $(1,0,0,0),$ $(1/3,0,1/3,1),$ $(5/8,5/8,1/8,1/2),$ $(1/2,0,0,1),$\\ $(1/4,1/4,1/4,1),$
$(1/3,1/3,0,1),$ $(1,0,0,1/2),$ $(2/3,2/3,0,1/2),$ $(2/3,2/3,0,0),$ $(0,1/2,0,1),$\\
 $(0,1/3,1/3,1),$ $(0,0,0,1),$ $(0,5/6,1/3,1/2),$ $(0,0,1/2,1),$ $(0,0,0,0),$ $(0,2/3,2/3,0).$
\end{lemma}

\begin{proof} 
To deduce the first inequality,  we test \eqref{DeltaPPU} with $f_1=f_2=f_4= 1_{S_t}$ and $f_3=\delta_0.$
To obtain the second one,  we test \eqref{DeltaPPU} with $f_1=f_3=1_{S_t}, f_2=\delta_0, $ and $f_4=1_{\mathbb F_q^d}.$
To get the third one, we test \eqref{DeltaPPU} with $f_1=\delta_0, f_2=f_3=1_{S_t}, $ and $f_4=1_{\mathbb F_q^d}.$
To prove the fourth one, we test \eqref{DeltaPPU} with $f_1=f_3=1_{S_t}$ and $ f_2=f_4=\delta_0.$
Finally, to obtain the fifth inequality, we test \eqref{DeltaPPU} with $f_1=f_4=\delta_0$ and $f_2=f_3=1_{S_t}.$
\end{proof}

\subsection{Sufficient conditions for the boundedness of $\unlhd$ on $\mathbb F_q^d$}
When one of exponents $p_1, p_2, p_4$ is $\infty$,  the boundedness problem of $\unlhd(p_1, p_2, p_3,  p_4)$ can be reduced to that for the $K_3$-Operator $\Delta$ or the $P_2$-Operator $\Lambda.$  
\begin{proposition}\label{RelationDL}  Let $1\le a, b,c \le \infty.$ 
\begin{enumerate}
\item  [(i)]If $\Delta(a,b,c)\lesssim 1$, then  $\unlhd(a,b, c, \infty)\lesssim 1.$
\item [(ii)] If $\Lambda(a,b,c)\lesssim 1$, then  $\unlhd(\infty, a, b,c) \lesssim 1$ and $\unlhd(a, \infty, b,c)\lesssim 1.$
\end{enumerate}
\end{proposition}

\begin{proof} For all non-negative functions $f_i, i=1,2,3,4,$ on $\mathbb F_q^d.$  we aim to prove the following inequalities:
\begin{equation}\label{fir1} \unlhd(f_1, f_2,f _3, f_4) \lesssim \left\{\begin{array}{ll} & \Delta(f_1,f_2, f_3) ~||f_4||_\infty, \\
                                                                                 & ||f_1||_\infty ~\Lambda(f_2 ,f_3, f_4),\\
                                                                                  & ||f_2||_\infty ~\Lambda(f_1, f_3, f_4).\end{array} \right.\end{equation}

 By the definition,  $\unlhd(f_1, f_2,f _3, f_4)$ can be expressed as
$$ \frac{1}{q^d |S_t||S_t^{d-2}|} \sum_{x^1, x^2, x^3\in \mathbb F_q^d} S_t(x^1-x^2) S_t(x^2-x^3)  S_t(x^3-x^1) \left(\prod_{i=1}^3 f_i(x^i)\right) \left( \frac{1}{|S_t|} \sum_{x_4\in \mathbb F_q^d} S_t(x^3-x^4) f_4(x^4)\right).$$
The sum in the above bracket  is clearly dominated by $||f_4||_\infty$ for all $x^3\in \mathbb F_q^d.$ Hence, recalling the definition of $\Delta(f_1,f_2,f_3)$  in \eqref{KDeltatD},   we get the first inequality  in \eqref{fir1}:
$$ \unlhd(f_1, f_2,f _3, f_4) \le \Delta(f_1,f_2,f_3) ||f_4||_\infty.$$

Now we prove the second and third inequalities in \eqref{fir1}.  We will only provide the proof of the second inequality, that is
\begin{equation}\label{unK}\unlhd(f_1, f_2,f _3, f_4) \le ||f_1||_\infty ~\Lambda(f_2 ,f_3, f_4).\end{equation}
The third inequality can be similarly proved by  switching the roles of variables $x^1, x^2.$   
We write $\unlhd(f_1, f_2, f_3, f_4)$ as follows:
$$ \frac{1}{q^d |S_t|^2} \sum_{x^2, x^3, x^4\in \mathbb F_q^d} S_t(x^2-x^3) S_t(x^3-x^4) \left(\prod_{i=2}^4f_i(x^i)\right)  \left( \frac{1}{|S_t^{d-2}|} \sum_{x^1\in \mathbb F_q^d} S_t(x^1-x^2) S_t(x^3-x^1) f_1(x^1)\right).$$
Recall the definition of $\Lambda(f_2, f_3, f_4)$ in \eqref{DefHinges}.   Then,  to prove the inequality \eqref{unK},  it will be enough to show that  for all $x^2, x^3 \in \mathbb F_q^d$ with $||x^2-x^3||=t\ne 0,$    the value in the above bracket  is $\lesssim ||f_1||_\infty.$ Now by a simple change of variables,  the value in the above bracket  is the same as
$$  \frac{1}{|S_t^{d-2}|} \sum_{x^1\in S_t}  S_t( (x^3-x^2)-x^1) f_1(x^1+x^2).$$
This is clearly dominated by
$$ \frac{1}{|S_t^{d-2}|} \sum_{x^1\in S_t}  S_t( (x^3-x^2)-x^1) ||f_1||_\infty.$$
 Since $||x^3-x^2||=t\ne 0, $    applying Corollary \ref{KKo} in Appendix gives us the desirable estimate.
\end{proof}

We address sufficient conditions for the boundedness of $\unlhd$ on $\mathbb F_q^d.$

The following result can be obtained from Proposition \ref{RelationDL} (i).
\begin{theorem}\label{MK1}
Let  $\unlhd$ be defined on the functions on $\mathbb F_q^d, d\ge 2.$ Suppose that  $1\le a, b\le \infty$ satisfies the following equations:
$$\frac{1}{a} +  \frac{d}{b}\le d \quad \mbox{and} \quad   \frac{d}{a} +  \frac{1}{b}\le d.$$
Then we have
$ \unlhd(a, b, \infty, \infty) \lesssim 1,  ~ \unlhd(a, \infty, b, \infty) \lesssim 1,~ \unlhd(\infty, a, b, \infty)\lesssim 1.$
\end{theorem}
\begin{proof} The statement follows immediately by combining   Proposition \ref{RelationDL} (i) with Theorem \ref{thmSi}.
\end{proof}

Proposition \ref{RelationDL} (ii) can be used to deduce the following result.
\begin{theorem} \label{MK2}
Let  $\unlhd$ be defined on the functions on $\mathbb F_q^d, d\ge 2.$  Suppose that 
$1\le a, b, c\le\infty$ satisfies one of the following conditions:
\begin{enumerate} 
\item [(i)] $0\le \frac{1}{a}, \frac{1}{c} \le \frac{d}{d+1}$ and $ \frac{1}{a}+\frac{d}{b}+ \frac{1}{c} \le d,$
\item [(ii)] $0\le \frac{1}{a} \le  \frac{d}{d+1} \le \frac{1}{c} \le 1,$ and $ \frac{1}{da}+\frac{1}{b}+ \frac{d}{c} \le d,$ 
\item [(iii)]  $0\le \frac{1}{c} \le  \frac{d}{d+1} \le \frac{1}{a} \le 1,$ and $ \frac{d}{a}+\frac{1}{b}+ \frac{1}{dc} \le d,$
\item [(iv)] $\frac{d}{d+1}\le \frac{1}{a}, \frac{1}{c} \le 1$ and $ \frac{d}{a}+\frac{1}{b}+ \frac{d}{c} \le 2d-1.$
\end{enumerate}
Then we have
$\unlhd(\infty, a, b, c)\lesssim 1~~ \mbox{and}\quad \unlhd(a, \infty, b, c)\lesssim 1.$
\end{theorem}
\begin{proof} From our assumptions on the numbers $a, b, c$,   Theorem \ref{SharpNS} implies that $\Lambda(a, b, c)\lesssim 1.$ Hence, the statement follows by applying   Proposition \ref{RelationDL} (ii).
\end{proof}

\subsection{Boundedness of $\unlhd$ in two dimensions}

Theorems \ref{MK1}, \ref{MK2}  provide  non-trivial results available in higher dimensions. In this section we will show that  further improvements can be made in two dimensions. 
Before we state and prove the improvements,  we collect the results in two dimensions, which can be direct consequences of Theorems \ref{MK1}, \ref{MK2}.

To deduce the following theorem, we will apply Theorem \ref{MK1} with $d=2.$ 
\begin{theorem}\label{Hb1} Let $\unlhd$  be defined on functions on $\mathbb  F_q^2.$ Then we have
$ \unlhd(p_1, p_2, p_3, p_4) \lesssim 1 $  provided that  $(p_1, p_2, p_3, p_4)$ is  one of the following points:
$(\infty, \infty, \infty, \infty),$ $(1, \infty, \infty, \infty),$ $(\infty, 1, \infty, \infty),$ $(\infty, \infty, 1, \infty),$
$(3/2, 3/2, \infty, \infty),$ $(3/2, \infty, 3/2, \infty),$ $(\infty, 3/2, 3/2, \infty).$
\end{theorem}
\begin{proof} Using the first conclusion of Theorem \ref{MK1} with $d=2$,  we see that $\unlhd(p_1, p_2, p_3, p_4) \lesssim 1 $ whenever $(p_1, p_2, p_3, p_4)$ takes the following points: 
$(\infty, \infty, \infty, \infty),$ $(1, \infty, \infty, \infty),$ $(\infty, 1, \infty, \infty),$ $(3/2, 3/2, \infty, \infty).$

Next,  the second conclusion of Theorem \ref{MK1} with $d=2$ implies that  $\unlhd(p_1, p_2, p_3, p_4) \lesssim 1 $ for the points $(p_1, p_2, p_3, p_4)=(\infty, \infty, 1, \infty),  (3/2, \infty, 3/2, \infty).$  Finally,  it follows from the third conclusion of Theorem \ref{MK1} with $d=2$ that   $\unlhd(p_1, p_2, p_3, p_4) \lesssim 1 $ for $(p_1, p_2, p_3, p_4)= (\infty, 3/2, 3/2, \infty).$ Hence, the proof is complete.
\end{proof}

The following theorem will be proven by applying Theorem \ref{MK2} with $d=2.$
\begin{theorem}\label{Hb2}
Let  $\unlhd$ be defined on the functions on $\mathbb F_q^2.$  Suppose that $(p_1, p_2, p_3, p_4)$ is one of the following points: $ (\infty, \infty, \infty, 1),$ $(2, \infty, 2, 2),$ $(3/2, \infty,  \infty, 3/2),$ $(\infty, \infty, 3/2, 3/2).$
Then we have 
$\unlhd(p_1, p_2, p_3, p_4) \lesssim 1. $
\end{theorem}
\begin{proof} We get that $\unlhd(\infty, \infty, \infty, 1) \lesssim 1$ by using the assumption (ii) and the first conclusion of Theorem \ref{MK2}.  Invoking the assumption (i) and the second conclusion of Theorem \ref{MK2},  one can directly note that  $\unlhd(2, \infty, 2, 2)\lesssim 1$ and $\unlhd(3/2, \infty,  \infty, 3/2)\lesssim 1.$ Finally, to prove that $\unlhd (\infty, \infty, 3/2, 3/2)\lesssim 1,$  one can use the assumption (i) and the first conclusion of Theorem \ref{MK2}.
\end{proof}

We now introduce the connection between $\unlhd(f_1,f_2, f_3, f_4)$  and the bilinear averaging operator.
\begin{proposition}\label{ABFormula}
Let $B$ be the bilinear operator defined as in \eqref{KdefB}.  Then, for any non-negative real-valued functions $f_i, i=1,2,3,4,$ on $\mathbb F_q^2,$  we have
$$\unlhd(f_1,f_2,f_3,f_4) = ||B(f_1, f_2)\cdot f_3 \cdot Af_4  ||_1,$$
where $A$ denotes the averaging operator over the circle in $\mathbb F_q^2.$
\end{proposition}
\begin{proof}  In two dimensions,   $\unlhd(f_1,f_2,f_3,f_4)$  can be rewritten as the following form:
$$\frac{1}{q^2 |S_t|^2 } \sum_{x^1, x^2, x^3, x^4\in \mathbb F_q^2} S_t(x^1-x^2) S_t(x^3-x^2) S_t(x^3-x^4) S_t(x^3-x^1) \prod_{i=1}^4 f_i(x^i). $$
By the change of variables by putting $y^1=x^3-x^1,  ~y^2=x^3-x^2, ~ y^3=x^3,~  y^4=x^3-x^4,$ the value $\unlhd(f_1, f_2,f _3, f_4)$ becomes
$$ \frac{1}{q^2 |S_t|^2} \sum_{y^1,y^2,y^3, y^4\in \mathbb F_q^2} S_t(y^2-y^1) S_t(y^2) S_t(y^4) S_t(y^1) f_1(y^3-y^1) f_2(y^3-y^2) f_3(y^3) f_4(y^3-y^4).$$
This can be expressed as follows:
$$ \frac{1}{q^2} \sum_{y^3\in \mathbb F_q^2} f_3(y^3)  \left(\frac{1}{|S_t|} \sum_{y^4\in S_t} f_4(y^3-y^4)\right)  \left(\frac{1}{|S_t|} \sum_{y^1, y^2\in S_t: ||y^2-y^1||=t}  f_1(y^3-y^1) f_2(y^3-y^2)\right).$$
Recalling the definitions of the averaging operator in \eqref{DefA}  and  the bilinear averaging operator in \eqref{KdefB},   it follows that
$$\unlhd(f_1,f_2,f_3,f_4) = \frac{1}{q^2} \sum_{y^3\in \mathbb F_q^2} f_3(y^3)  Af_4(y^3)  B(f_1, f_2)(y^3).$$

By the definition of the normalized norm $||~~||_1$,  the statement follows.                                                  
\end{proof}

 For $1\le p_1, p_2, p_3, p_4\le \infty$,  recall that the notation $\unlhd(p_1, p_2, p_3, p_4) \lessapprox 1$ is used if   the following estimate holds for all subsets  $E, F, G, H$ of $\mathbb F_q^2$:
$$  \unlhd(E,F,G,H) \lesssim ||E||_{p_1} ||F||_{p_2} ||G||_{p_3} ||H||_{p_4},$$ 
and this estimate is referred to as the restricted strong-type  $\unlhd(p_1, p_2, p_3, p_4)$ estimate.

The following theorem is our main result in two dimensions, which gives  a new restricted strong-type estimate for the boundedness on the operator $\unlhd$.
\begin{theorem}\label{RSTK}
Let $\unlhd$ be defined on functions on $\mathbb F_q^2.$  Let $1\le p_3, p_4\le \infty.$ Then the following statements are valid for all subsets $E, F$ of $\mathbb F_q^2$ and all non-negative functions $f_3, f_4$ on $\mathbb F_q^2.$
\begin{enumerate}
\item [(i)] If $2\le p_3\le \infty,  ~ 3/2 \le p_4 \le \infty,$ and  $\frac{1}{p_3}+\frac{1}{2p_4}\le \frac{1}{2},$ then  we have 
$$\unlhd(E, F, f_3, f_4)\lesssim ||E||_2||F||_2 ||f_3||_{p_3} ||f_4||_{p_4}.$$
\item [(ii)]  If $2\le p_3\le \infty, ~  4/3 \le p_4 \le 3/2,$ and  $\frac{1}{p_3}+\frac{2}{p_4}\le \frac{3}{2},$ then  we have 
$$\unlhd(E, F, f_3, f_4)\lesssim ||E||_2||F||_2 ||f_3||_{p_3} ||f_4||_{p_4}.$$
\end{enumerate}

\end{theorem}
\begin{proof} Let $E, F$ be subsets of $\mathbb F_q^2$ and $f, g$ be non-negative real-valued functions on $\mathbb F_q^2.$
By Proposition \ref{ABFormula} and   H\"older's inequality,  it follows that for $2\le p_3\le \infty,$ 
$$ \unlhd(E, F, f_3, f_4) \le || B(E, F)||_2  ||f_3||_{p_3} ||Af_4||_{\frac{2p_3}{p_3-2}}.$$
Here, we  also notice that $2\le \frac{2p_3}{p_3-2}\le \infty.$ 
Since  $||B(E, F)||_2 \lesssim  ||E||_2  ||F||_2$ by Lemma \ref{mainlemK},  we see that
$$ \unlhd(E, F, f_3, f_4) \le ||E||_2  ||F||_2  ||f_3||_{p_3} ||Af_4||_{\frac{2p_3}{p_3-2}}.$$
Hence, to complete the proof,  it suffices to show that for all exponents $p_3, p_4$ satisfying the assumptions of the theorem, we have
\begin{equation}\label{aimGo}  A\left(p_4 \to  \frac{2p_3}{p_3-2}\right) \lesssim 1.\end{equation}
To prove this,  we first recall  from Theorem \ref{sharpA} with $d=2$ that  $ A(p\to r) \lesssim 1$ for any numbers $1\le p, r\le \infty$ such that  $(1/p, 1/r)$ lies on the convex hull of points $(0,0), (0, 1),  (1,1), $ and $(\frac{2}{3},  \frac{1}{3} ).$
Also invoke Lemma \ref{BALem} to find the equations indicating the endpoint estimates for $A(p\to r)\lesssim 1.$
Using those averaging estimates with $p=p_4, r= \frac{2p_3}{p_3-2}$,  the inequality  \eqref{aimGo} can be obtained by a direct computation, where we also use the fact that $2\le r=\frac{2p_3}{p_3-2}\le \infty.$
\end{proof}

The following corollary is a direct consequence of Theorem \ref{RSTK}.
\begin{corollary}\label{Hb3} Let $\unlhd$ be defined on functions on $\mathbb F_q^2.$ Then we have
$ \unlhd(2,2, 2, \infty)\lessapprox 1.$
\end{corollary}
\begin{proof}
The statement follows by a direct application of  Theorem \ref{RSTK} (i).
\end{proof}

The theorem below shows that   the exponents for $\diamondsuit_t(p_1, p_2, p_3, p_4)\lesssim 1$  are more restricted than those for $\unlhd(p_1, p_2, p_3, p_4)\lesssim 1$ up to the endpoints. This also provides a positive answer to Question \ref{PConj} since the graph $K_3$ + tail is a subgraph of the graph $C_4$ + diagonal. 
\begin{theorem} \label{TK} Let $\diamondsuit_t$ and $\unlhd$ be the operators acting on the functions on $\mathbb F_q^2.$
Suppose that  $\diamondsuit_t(p_1, p_2, p_3, p_4)\lesssim 1$ for $1\le p_1, p_2, p_3, p_4\le \infty.$ 
Then we have $\unlhd(p_1, p_2, p_3, p_4)\lesssim 1$ except for the point $(2,2,2, \infty).$  In addition, we have $\unlhd(2,2, 2, \infty)\lessapprox 1.$ 
\end{theorem}
\begin{proof}
Assume that $\diamondsuit_t(p_1, p_2, p_3, p_4)\lesssim 1$ for $1\le p_1, p_2, p_3, p_4\le \infty.$ 
Then, by Lemma \ref{N2244}, the point  $(1/p_1, 1/p_2, 1/p_3, 1/p_4)$ is contained in the convex hull of  the following points:
$(0,0,1,0),$ $(0,1,0,0),$ $(0,0,0,1),$ $(1/2,0,1/2,1/2),$ $(2/3,2/3,0,0),$ $(1,0,0,0),$
$(2/3,0,2/3,0),$ $(1/2,1/2,1/2,0),$ $(2/3,0,0,2/3),$ $(0,2/3,2/3,0),$ $(0,0,0,0),$ $(0,0,2/3,2/3).$

By Theorems \ref{Hb1}, \ref{Hb2},  the strong type estimate $\unlhd(p_1, p_2, p_3, p_4)\lesssim 1$ holds for all the above points $(1/p_1, 1/p_2, 1/p_3, 1/p_4)$ except for $(1/2, 1/2,1/2, 0).$  Moreover,  we know from Corollary \ref{Hb3} that $\unlhd(2,2,2, \infty)\lessapprox 1.$ Hence, the statement follows by interpolating  those points.
\end{proof}

\begin{remark}\label{ReRk}
The reverse statement of Theorem \ref{TK} is not true. To see this, observe from  Theorem \ref{RSTK} (ii) that  $\unlhd(2,2, 6, 3/2)\lessapprox 1.$ In addition, by Lemma \ref{N2244}, notice that $\diamondsuit_t(2,2, 6, 3/2)$ cannot be bounded.
\end{remark}
\begin{corollary}\label{CorTK} 
Conjecture \ref{PConjK} holds up to endpoints for the graph $C_4$ + diagonal and its subgraph  $K_3$ + tail in $\mathbb F_q^2.$
\end{corollary}
\begin{proof}
The operators $\diamondsuit_t$ and $\unlhd$ are associated with the $C_4$ + diagonal and its subgraph  $K_3$ + tail in $\mathbb F_q^2,$ respectively. Hence, invoking Theorem \ref{TK}, the proof is  reduced to showing that  the $C_4$ + diagonal and its subgraph  $K_3$ + tail satisfy the main hypothesis \eqref{mainConj} of Conjecture \ref{PConjK}. However, it is clear that
$$ \min\{\delta(C_4 + diagonal),  2\}=2 > 1=\delta(K_3 + tail).$$
Thus, the proof is complete.
\end{proof}

The following theorem shows that there exists  an inclusive relation between  boundedness exponents for the operators corresponding to the graphs $C_4$ and $K_3$ + tail, although they are not subgraphs of each other.
\begin{theorem} \label{DU}Let $\Diamond$ and $\unlhd$ be defined on functions on $\mathbb F_q^2$ and let $1\le p_1, p_2, p_3, p_4\le \infty.$ Then if $\Diamond(p_1, p_2, p_3, p_4) \lesssim 1$,  we have $\unlhd(p_1, p_2, p_3, p_4)\lesssim 1.$
\end{theorem}
\begin{proof} First, by Theorem \ref{Thm7.6K}, note that $\Diamond(2,2,2, \infty)$ cannot be bounded.
Now suppose that $\Diamond(p_1, p_2, p_3, p_4)\lesssim 1.$ Then $(p_1, p_2, p_3, p_4)\ne (2,2,2,\infty).$  
Using Theorem \ref{CExa}, we get $\diamondsuit_t(p_1, p_2, p_3, p_4)\lesssim 1.$  Then the statement follows immediately from Theorem \ref{TK}.
\end{proof}

By combining Remark \ref{ReRk} and Theorem \ref{CExa}, it is clear that the reverse of Theorem \ref{DU} does not hold. 
Notice that  Theorem \ref{DU} provides an example to satisfy Conjecture \ref{PConjK} without the hypothesis that $G'$ is a subgraph of the graph $G.$

\section{Boundedness problems for the $Y$-shaped graph}
In this section, we study the boundedness of the operator for the $Y$-shaped graph in Figure \ref{fig:YS}. For $t\in \mathbb F_q^*,$ the $Y$-shaped operator $Y$ is defined by
\begin{equation}\label{defY} Y(f_1, f_2, f_3, f_4)= \frac{1}{q^d |S_t|^3}\sum_{x^1, x^2, x^3, x^4\in \mathbb F_q^d} S_t(x^3-x^1) S_t(x^3-x^2) S_t(x^3-x^4)  \prod_{i=1}^4 f_i(x^i),\end{equation}
where functions $f_i, i=1,2,3, 4,$ are defined  on $\mathbb F_q^d.$
Note that this operator $Y$ is related to the $Y$-shaped graph, and so the normalizing factor $\mathcal{N}(G)$ in \eqref{generalform} can be taken as  $q^d |S_t|^3.$ \\

The operator norm of the $Y$-shaped operator is defined in a standard way.
\begin{definition} For $1\le p_1, p_2, p_3, p_4\le \infty,$ 
the operator norm $Y(p_1, p_2, p_3, p_4)$ is defined as  the smallest constant such that   the following estimate holds  for all non-negative real-valued functions $f_i, i=1,2,3, 4,$ on $\mathbb F_q^d:$
\begin{equation}\label{Ytest} Y(f_1,f_2, f_3, f_4) \le Y(p_1, p_2, p_3, p_4) ||f_1||_{p_1} ||f_2||_{p_2} ||f_3||_{p_3} ||f_4||_{p_4}.\end{equation}
\end{definition} 

We aim to find  all numbers $1\le p_1, p_2, p_3, p_4 \le \infty$ such that $ \unlhd(p_1, p_2, p_3, p_4) \lesssim 1.$

\begin{lemma} [Necessary conditions for the boundedness of $Y(p_1, p_2, p_3, p_4)$]\label{NeY} Let $1\le p_i \le \infty, 1\le i\le 4.$ Suppose that $Y(p_1, p_2, p_3, p_4) \lesssim 1.$ Then all the following inequalities are satisfied:\\
$ \frac{1}{p_1} +\frac{1}{p_2}+\frac{d}{p_3}+\frac{1}{p_4} \le d,$~
$ \frac{d}{p_1} +\frac{d}{p_2}+\frac{1}{p_3}+\frac{d}{p_4} \le 3d-2,$ ~
 $\frac{d}{p_1} +\frac{d}{p_2}+\frac{1}{p_3}\le 2d-1,$
 ~ $\frac{d}{p_1} +\frac{1}{p_3}+\frac{d}{p_4}\le 2d-1,$
 ~ $\frac{d}{p_2} +\frac{1}{p_3}+\frac{d}{p_4}\le 2d-1,$
~$\frac{d}{p_1} +\frac{1}{p_3}\le d,$ 
~$\frac{d}{p_2} +\frac{1}{p_3}\le d, $
~$\frac{1}{p_3} +\frac{d}{p_4}\le d.$

In particular, if $d=2$, then it can be shown by Polymake \cite{AGHJLP, GJ} that $(1/p_1, 1/p_2, 1/p_3, 1/p_4)$ is contained in the convex hull of the  points: $(0,0,1,0),$ $(1,0,0,1/2),$  $(1, 0, 1/2, 0),$ $(1/2, 1, 0, 1/2),$ $(1, 1/2, 0, 1/2),$ $(1/2, 1/2, 0, 1),$ $(1, 0, 1/3, 1/3),$ $(1/2, 5/6, 1/3, 0),$ 
 $(1,1/2, 0, 0),$ $(1/2, 0, 1/3, 5/6),$\\ $(1/2,1,0,0),$
  $(1, 1/3, 1/3,0),$ $(1/2, 0,0,1),$ $(1,0,0,0),$  $(0,0,0,0),$ $(0, 0, 0, 1),$ $(0, 1, 0, 0),$ $(0, 0, 2/3, 2/3),$ $(0, 5/6, 1/3, 1/2),$ $(0, 1, 0, 1/2),$ $(0,1/2, 1/3, 5/6),$ $(0, 1/2, 0, 1),$ $(0,2/3,2/3,0).$
\end{lemma}

\begin{proof} 
By a direct computation,  the conclusions of the lemma easily follow by testing the inequality \eqref{Ytest} with the following specific functions, respectively:
\begin{tabbing}
\=1) \hspace{4cm} \=2) \hspace{4cm} \kill
1) $f_1=f_2=f_4=1_{S_t}, ~ f_3=\delta_0.$\hspace{16mm}
2) $f_1=f_2=f_4=\delta_0, ~f_3=1_{S_t}.$\\
3) $f_1=f_2=\delta_0, f_3=1_{S_t}, f_4=1_{\mathbb F_q^d}.$\hspace{10mm}
4) $f_1=f_2=\delta_0, f_2=1_{\mathbb F_q^d}, f_3=1_{S_t}.$\\
5) $f_1=1_{\mathbb F_q^d}, f_2=f_4=\delta_0, f_3=1_{S_t}.$\hspace{10mm}
6) $f_1=\delta_0, f_2=f_4=1_{\mathbb F_q^d}, f_3=1_{S_t}.$ \\
7)  $f_1=f_4=1_{\mathbb F_q^d}, f_2=\delta_0, f_3=1_{S_t}.$\hspace{10mm}
8)  $f_1=f_2=1_{\mathbb F_q^d}, f_3=1_{S_t}, f_4=\delta_0.$
\end{tabbing}
\end{proof}

\subsection{Sufficient conditions for the boundedness of $Y$ on $\mathbb F_q^d$}
It is not hard to observe that the boundedness problem for the $Y$-shaped operator can be reduced to the spherical averaging estimate. Indeed, the value $Y(f_1,f_2, f_3, f_4)$ in \eqref{defY} can be written by
$$Y(f_1,f_2, f_3, f_4)=\frac{1}{q^d} \sum_{x^3\in \mathbb F_q^d} f_3(x^3) \prod_{i=1,2, 4}\left(\frac{1}{|S_t|} \sum_{x^i\in \mathbb F_q^d} S_t(x^3-x^i) f_i(x^i)\right).$$
Invoking the definition of the averaging operator $A=A_{S_t}$ in \eqref{AAdef},  we get
$$ Y(f_1, f_2, f_3, f_4)= \frac{1}{q^d} \sum_{x^3\in \mathbb F_q^d} f_3(x^3) Af_1 (x^3) Af_2(x^3) Af_4(x^4) =|| Af_1 \cdot Af_2 \cdot f_3\cdot Af_4||_1.$$
By H\H{o}lder's inequality and the nesting property of the norm $||\cdot ||_p,$  we get 
\begin{equation}\label{IKBP}
Y(f_1, f_2, f_3, f_4) \le ||Af_1||_{r_1} ||Af_2||_{r_2} ||f||_{p_3} ||Af_4||_{r_4} \quad\mbox{if}\quad \frac{1}{r_1}+\frac{1}{r_2} + \frac{1}{p_3}+ \frac{1}{r_4} \le 1.
\end{equation}
\begin{proposition} \label{proY}
Let $1\le p_1, p_2, p_3, p_4, r_1, r_2, r_4 \le \infty$ be extended real numbers which satisfy the following assumptions:
$\frac{1}{r_1}+\frac{1}{r_2} + \frac{1}{p_3}+ \frac{1}{r_4} \le 1$ and $A(p_i\to r_i)\lesssim 1 $ for all $i=1,2, 4.$ Then we have 
$$Y(p_1, p_2, p_3, p_4)\lesssim 1.$$  
\end{proposition}
\begin{proof}
By combining the inequality \eqref{IKBP} with our assumptions on the averaging estimates,  it follows that for all functions $f_i, i=1,2,3, 4,$ on $\mathbb F_q^d,$ 
$$ Y(f_1,f_2,f_3, f_4) \lesssim ||f_1||_{p_1} ||f_2||_{p_2} ||f_3||_{p_3} ||f_4||_{p_4}.$$
This completes the proof.
\end{proof}

The following theorem provides lots of sufficient conditions for the boundedness of the $Y$-shaped operator. 

\begin{theorem} \label{thm10.1} Let  $1\le p_1,p_2, p_3, p_4\le \infty,$ and $Y$ be the $Y$-shaped operator on $\mathbb F_q^d.$ Then $Y(p_1, p_2, p_3, p_4)\lesssim 1$ provided that  one of the following conditions is satisfied:
\begin{enumerate}
\item [(i)] $0\le \frac{1}{p_1}, \frac{1}{p_2}, \frac{1}{p_4}\le \frac{d}{d+1}$ and $  \frac{1}{p_1}+\frac{1}{p_2}+\frac{d}{p_3}+\frac{1}{p_4} \le d.$

\item [(ii)] $0\le \frac{1}{p_1}, \frac{1}{p_2} \le \frac{d}{d+1} \le \frac{1}{p_4}$ and $  \frac{1}{dp_1}+\frac{1}{dp_2}+\frac{1}{p_3}+\frac{d}{p_4} \le d.$

\item [(iii)] $0\le \frac{1}{p_1}, \frac{1}{p_4}\le \frac{d}{d+1} \le \frac{1}{p_2}\le 1$ and $  \frac{1}{dp_1}+\frac{d}{p_2}+\frac{1}{p_3}+\frac{1}{dp_4} \le d.$
 
\item [(iv)] $0\le \frac{1}{p_2},  \frac{1}{p_4} \le \frac{d}{d+1} \le \frac{1}{p_1} \le 1$ and $  \frac{d}{p_1}+\frac{1}{dp_2}+\frac{1}{p_3}+\frac{1}{dp_4} \le d.$

\item [(v)] $0\le \frac{1}{p_1}\le \frac{d}{d+1} \le \frac{1}{p_2}, \frac{1}{p_4}\le 1$ and $\frac{1}{dp_1}+\frac{d}{p_2}+\frac{1}{p_3}+\frac{d}{p_4} \le 2d-1.$

\item [(vi)] $0\le \frac{1}{p_2}\le \frac{d}{d+1} \le \frac{1}{p_1},\frac{1}{p_4}\le 1$  and $ \frac{d}{p_1}+\frac{1}{dp_2}+\frac{1}{p_3}+\frac{d}{p_4} \le 2d-1.$

\item [(vii)] $0\le \frac{1}{p_4}\le \frac{d}{d+1}\le \frac{1}{p_1}, \frac{1}{p_2} \le 1$  and 
 $  \frac{d}{p_1}+\frac{d}{p_2}+\frac{1}{p_3}+\frac{1}{dp_4} \le 2d-1.$

\item [(viii)]  $\frac{d}{d+1}\le \frac{1}{p_1}, \frac{1}{p_2}, \frac{1}{p_4}\le 1$  and 
 $  \frac{d}{p_1}+\frac{d}{p_2}+\frac{1}{p_3}+\frac{d}{p_4} \le 3d-2.$
\end{enumerate}
\end{theorem}
\begin{proof}
The proof uses Proposition \ref{proY}  and the sharp averaging estimates in Lemma \ref{BALem}. The proof of this theorem is similar to that of Theorem \ref{SharpNS}. Therefore, we leave the detail of the proof to readers.
\end{proof}
%%%%%%%%%%%%%%%%%%%%%%%%%%%%%%%%%%%%%%%%%%%%%%
Conjecture \ref{PConjK} is also supported by the following theorem.
%%%%%%%%%%%%%%%%%%%%%%%%%%%%%%%%%%%%%%%%%%%
\begin{theorem}\label{ThmDtY} Let ${\diamondsuit}_t$ and $Y$ be the operators acting on the functions on $\mathbb F_q^2.$
If  ${\diamondsuit}_t(p_1, p_2, p_3, p_4)\lesssim 1$ with $1\le p_1, p_2, p_3, p_4\le \infty,$ 
then  $Y(p_1, p_2, p_3, p_4)\lesssim 1.$
\end{theorem}
\begin{proof}
Assume that ${\diamondsuit}_t(p_1, p_2, p_3, p_4)\lesssim 1.$ Then, by Lemma \ref{N2244}, 
$(1/p_1, 1/p_2, 1/p_3, 1/p_4)$ is contained in the convex hull of the points $(0,0,1,0),$ $(0,1,0,0),$ $(0,0,0,1),$ $(1/2,0,1/2,1/2),$ $(2/3,2/3,0,0),$ $(1,0,0,0),$ $(2/3,0,2/3,0),$ $(1/2,1/2,1/2,0),$ $(2/3,0,0,2/3),$ $(0,2/3,2/3,0),$ $(0,0,0,0),$ $(0,0,2/3,2/3).$
By interpolating those critical points, it suffices to check that each critical point above satisfies  one of the 8 hypotheses of Theorem \ref{thm10.1} with $d=2$. However, this can be easily shown by a direct computation.  For example,  for the critical point $(1/p_1, 1/p_2, 1/p_3, 1/p_4)=(1/2,1/2,1/2,0)$,  we can invoke the hypothesis (i) of Theorem \ref{thm10.1} with $d=2$ and obtain that $Y(2,2,2,\infty)\lesssim 1.$ In a same way, it can be easily proven for other critical points. 
\end{proof}

\begin{remark} \label{remark10.6K} The reverse statement of Theorem \ref{ThmDtY} is not true. To find a counterexample, we can take $p_1=p_3=\infty$, $p_2=p_4=3/2.$ Indeed, by the hypothesis (5) of Theorem \ref{thm10.1} with $d=2$,  we see that 
$Y(\infty, 3/2,\infty, 3/2) \lesssim 1.$ However,  ${\diamondsuit}_t(\infty, 3/2,\infty, 3/2)$ is not bounded, which follows from Lemma \ref{N2244} with $d=2.$ 
\end{remark}

The following corollary proposes some possibility that the assumption of the subgraph in Conjecture \ref{PConjK} can be dropped.
\begin{corollary} \label{CorThmDtY}
Let $\Diamond$ and $Y$ be the operators acting on the functions on $\mathbb F_q^2.$
If  $\Diamond(p_1, p_2, p_3, p_4)\lesssim 1$ with $1\le p_1, p_2, p_3, p_4\le \infty,$ 
then  $Y(p_1, p_2, p_3, p_4)\lesssim 1.$
\end{corollary}
\begin{proof}
The statement of the corollary follows immediately by combining Theorem \ref{CExa} and Theorem \ref{ThmDtY}.
\end{proof}

Combining all the results obtained so far, we get the following theorem:
\begin{theorem} \label{mainthm22} When $d=2$ and $n=3, 4$,  Conjecture \ref{PConjK} is true, where we accept boundedness results up to endpoints in the case when $G$ is the $C_4$ + diagonal and its subgraph $G'$ is the $K_3 + tail.$ 
\end{theorem}
\begin{proof} By  Corollary \ref{CormainthmS5} for $n=3,$ and   by Corollaries \ref{CorThmDtCap}, \ref{CorThmBoxK}, \ref{CorTK}, \ref{CorThmDtY} for $n=4,$
we have proven that for $d=2$ and $n = 3,4,$ there is the required inclusive boundedness relationship between any two operators corresponding to arbitrary connected ordered graph $G$ and its subgraph $G'$ except for the following three cases:
\begin{itemize}
\item [(I)] $G=C_4$ + diagonal and $G'=C_4.$
\item [(II)] $K_3$ + tail and $G'=Y$-shape.
\item [(III)] $G=K_3$ + tail and $G'=P_3.$
\end{itemize}
However, since $\delta(G)=\delta(G')$ for each case of (I), (II), (III), they do not satisfy the main hypothesis \eqref{mainConj}  of  Conjecture \ref{PConjK}.
Hence, they cannot be counterexamples contradicting Conjecture \ref{PConjK} and so there is no counterexample against Conjecture \ref{PConjK}, as required.
\end{proof}

\section{Acknowledgements}
A. Iosevich and P. Bhowmik were supported in part by the National Science Foundation grant no. HDR TRIPODS--1934962 and the National Science Foundation grant DMS--2154232. D. Koh was supported by  Basic Science Research Programs through National Research Foundation of Korea (NRF) funded by the Ministry of Education (NRF-2018R1D1A1B07044469).
 T. Pham would like to thank to the VIASM for the hospitality and for the excellent working conditions.

\section{Appendix}
In this appendix, we introduce the number of intersection points of two spheres in $\mathbb F_q^d.$
Let $\eta$ denote the quadratic character of $\mathbb F_q^*$, namely,  $\eta(s)=1$ for a square number $s$ in $\mathbb F_q^*$, and $\eta(s)=-1$ otherwise.
\begin{definition}
Given a non-zero vector $m$ in $\mathbb F_q^d, $ and $t, b\in \mathbb F_q,$  we define  $N(m,t,b)$ to be the number of  common solutions $x\in \mathbb F_q^d$ of the following equations:
$||x||=t,   ~~ m\cdot x=b.$ 
\end{definition}

Notice that the value of $N(m,t,b)$ is the number of all intersection points between the sphere $S_t$ and the plane $\{x\in \mathbb F_q^d: m\cdot x=b\}.$  The explicit value of it is well known as follows. 
\begin{lemma}\label{NSolution} Let $b, t\in \mathbb F_q,$ and let $m$ be a non-zero element in $\mathbb F_q^d, d\ge 2.$ Then  the following statements hold:
\begin{enumerate}
\item [(i)]If $||m||\ne 0$ and $  b^2-t||m||=0,$ then 
$$ N(m, t,b)= \left\{\begin{array}{ll} q^{d-2} \quad &\mbox{if} ~~ d~\mbox{is even,}\\
                                                 q^{d-2}+ q^{\frac{d-3}{2}} (q-1) \eta\left((-1)^{\frac{d-1}{2}} ||m||\right) &\mbox{if} ~~ d~\mbox{is odd}.\end{array}\right.$$                                            
\item [(ii)]If $||m||\ne 0$ and $  b^2-t||m||\ne 0,$ then 
$$ N(m, t,b)= \left\{\begin{array}{ll} q^{d-2}+q^{\frac{d-2}{2}}\eta\left((-1)^{\frac{d}{2}} (b^2-t||m||)\right)   \quad &\mbox{if} ~~ d~\mbox{is even,}\\
                                                 q^{d-2}- q^{\frac{d-3}{2}} \eta\left((-1)^{\frac{d-1}{2}} ||m||\right) &\mbox{if} ~~ d~\mbox{is odd}.\end{array}\right.$$

   \item [(iii)]If $||m||=0=  b^2-t||m||,$ then 
$$ N(m, t,b)= \left\{\begin{array}{ll} q^{d-2}+ \nu(t) q^{\frac{d-2}{2}}\eta\left((-1)^{\frac{d}{2}})\right)   \quad &\mbox{if} ~~ d~\mbox{is even,}\\
                                                 q^{d-2}- q^{\frac{d-1}{2}} \eta\left((-1)^{\frac{d-1}{2}} t\right) &\mbox{if} ~~ d~\mbox{is odd},\end{array}\right.$$
where  $\nu(t)=-1$ if $t\in \mathbb F_q^*$ and $\nu(0)=q-1.$                                                                                
                                                                                 
\item [(iv)] If $||m||= 0$ and $  b^2-t||m||\ne 0,$ then  $ N(m, t,b)=  q^{d-2}.$

\end{enumerate}
\end{lemma}
\begin{proof} See  Exercises 6.31--6.34 in \cite{LN97},  or  one can prove it by using the discrete Fourier analysis with the explicit value of the Gauss sum. 
\end{proof}

By a direct application of Lemma \ref{NSolution},  one can find the explicit number of the intersections of two spheres over finite fields.
Precisely we have the following result.
\begin{theorem}\label{KCor5.5}
Given a non-zero vector $m\in \mathbb F_q^d$ and  $ t, j\in \mathbb F_q,$  let 
$$\Theta(m, t, j):=\{x\in S_t: ||x-m||=j\}|.$$
If $m\in S_\ell, $ then 
$|\Theta(m, t, j)|=N\left(m, t,  \frac{t+\ell-j}{2}\right).$
\end{theorem}    
\begin{proof}
Since $||x-m||=t+\ell-2m\cdot x$ for $x\in S_t, m\in S_\ell,$  it is clear that $\Theta(m, t, j)$ is the number of common solutions $x$ of the following equations:
$$ ||x||=t,  \quad m\cdot x=\frac{t+\ell-j}{2}.$$
Hence, by the definition of $N$, we obtain the required conclusion.
\end{proof}

\begin{corollary} \label{KKo} Let $t\in \mathbb F_q^*$ and $ \ell\in \mathbb F_q.$  
Then, for every non-zero vector $m\in S_\ell,$ we have
$$\sum_{x\in S_t: ||x-m||=t} 1 \sim q^{d-2}$$
excepting for the following three cases:

\begin{tabbing}
\=1) \hspace{4cm} \=2) \hspace{4cm} \kill
1)  $d=2, \ell\ne 0,  \eta (t\ell-\ell^2/4)=-1.$ \hspace{3mm}
2)  $d=2, \ell=0,  \eta(-1)=1.$\\
3) $d=3, \ell=0,  \eta(-t)=1.$ 
\end{tabbing}

For each of those three cases, the value in the above sum takes  zero.
On the other hand,  if $d=2, \ell\ne 0, $ and $\eta (t\ell-\ell^2/4)=1$,   the value in the above sum  is exactly two.
\end{corollary}

\begin{proof}
It follows from Theorem \ref{KCor5.5} that  for any $||m||=\ell,$ 
$$ \sum_{x\in S_t: ||x-m||=t} 1 = N\left(m, t, \frac{\ell}{2}\right),$$
and so  the corollary  is a direct consequence of  Lemma \ref{NSolution} (i), (ii), (iii).
\end{proof}
\bibliographystyle{amsplain}

%\begin{lemma} \label{NumberN}
%For $u\in S_t \subset \mathbb F_q^d, d\ge 2,$ let $N(u)$ be the number of the vectors $v\in S_t$ satisfying that $v-u\in S_t.$ Then, for any $u$ in $S_t, t\ne0,$  we have
%$$ N(u)=\left\{\begin{array}{ll} q^{d-2} + q^{(d-2)/2} \eta\left(3(-1)^{(d+2)/2}\right)\quad&\mbox{if} ~~d~~\mbox{is even,}\\ 
% q^{d-2} - q^{(d-3)/2} \eta\left((-1)^{(d-1)/2}t\right)\quad&\mbox{if} ~~d~~\mbox{is odd},\end{array}\right.
%$$
%where $\eta$ denotes the quadratic character of $\mathbb F_q$.

%In particular, if $d=2$, then $N(u)=2$ for $\eta(3)=1$, and $0$ otherwise. If $d\ge 3,$ then $N(u)>0.$
%\end{lemma}

\end{document}